\documentclass[11pt,a4paper]{article}
\usepackage[T1]{fontenc}
\usepackage[utf8]{inputenc}
\usepackage{enumitem}
\usepackage{float}
\usepackage{bm}
\usepackage{algpseudocode}
\usepackage{subcaption} 
\usepackage{mathrsfs,amssymb,amsmath,amsfonts}
\usepackage{amsthm}
\usepackage{wrapfig} 
\usepackage{epstopdf}
\usepackage{graphicx}
\usepackage{epsfig}
\usepackage{algorithm}
\usepackage{hyperref}
\usepackage{cite}
\usepackage{caption}
\usepackage[page]{appendix}

\newtheorem{assumption}{Assumption}[section]
\newtheorem{theorem}{Theorem}[section]
\newtheorem{lemma}[theorem]{Lemma}
\newtheorem{definition}[theorem]{Definition}

\theoremstyle{definition}

\newenvironment{keywords}{{\bf Key words: }}{}

\begin{document}

\title{Model-Free Q-Learning for Infinite-Horizon Stochastic Linear Quadratic Problems with Regime Switching}

\author{Xinyue Zhang\footnote{School of Statistics and Mathematics, Shandong University of Finance and Economics, Jinan 250014, China. Email: {\tt xinyuezhang1022@163.com}.}, \and 
Na Li\footnote{Corresponding author. School of Mathematical Sciences, Dalian University of Technology, Dalian 116024, China. Email: {\tt lina2025@dlut.edu.cn}.}, \and
Xun Li\footnote{Department of Applied Mathematics, The Hong Kong Polytechnic University, Hong Kong SAR, China. Email: {\tt li.xun@polyu.edu.hk}.}, \and 
Zuo Quan Xu\footnote{Department of Applied Mathematics, The Hong Kong Polytechnic University, Hong Kong SAR, China. Email: {\tt maxu@polyu.edu.hk}.
The work of Na Li is supported in part by the National Natural Science Foundation of China (12571475 and 12171279).
The work of Xun Li is supported in part by the Research Grants Council of Hong Kong (15225124), and in part by the PolyU
(4-ZZVB). The work of Zuo Quan Xu is supported in part by the National Natural Science Foundation of China (12571517), 
in part by the Hong Kong RGC (15203423 and 15204622), in part by the PolyU-SDU Joint Research Center on Financial Mathematics, 
in part by the CAS AMSS-PolyU Joint Laboratory of Applied Mathematics, in part by the Research Centre for Quantitative Finance (1-CE03), 
and in part by the internal grants from The Hong Kong Polytechnic University.}}

\maketitle

\begin{abstract}
This paper addresses  infinite-horizon continuous-time stochastic linear quadratic optimal control problems with regime switching. 
We propose a paradigm shift from model-based design by adopting an adaptive dynamic programming approach, specifically developing on-policy and 
off-policy Q-learning algorithms that learn the optimal controller solely from online state trajectory data. The theoretical core of our work 
consists of a complete proof of the equivalence between the on- and off-policy architectures, alongside a rigorous analysis establishing the stability 
of the closed-loop system and the convergence of the algorithms to the optimal solution. For computational tractability, we implement these algorithms 
using vectorization and Kronecker product algebra. The theoretical results are corroborated by numerical case studies that clearly demonstrate the 
operational effectiveness and practical feasibility of the proposed model-free control strategy.
\end{abstract}

\begin{keywords}
Q-learning, Infinite-horizon, LQ problems with regime switching, Stochastic optimal control.
\end{keywords}\\

\section{Introduction}
~~~Many practical control systems are subject to abrupt structural changes caused by random environmental fluctuations. Such phenomena are naturally captured by regime switching model, where the evolution of the system modes is governed by an underlying Markov chain. 
Owing to its ability to model complex, randomly switching dynamics, regime switching model has found widespread application across diverse engineering domains, including power systems \cite{S Kuppusamy}, robotics \cite{B Jiang}, and wind energy conversion \cite{P. Cheng}. 
Beyond engineering, the framework has also proven instrumental in financial mathematics, for instance, Zhou and Yin \cite{X. Y. Zhou} developed a continuous-time Markowitz model for portfolio optimization under regime-switching market conditions. 
The study of regime switching model thus carries both profound theoretical significance and substantial practical utility, offering a rigorous foundation for tackling challenging stochastic control problems in complex environments.

Over the past two decades, the linear quadratic (LQ) control problem with regime switching has attracted considerable research attention, yielding a rich body of literature (e.g., \cite{M. D. Fragoso, Dragan V, X Zhang, N. Li}). The prevailing approach in this domain hinges on 
solving the coupled algebraic Riccati equations (CAREs) for a unique positive definite solution (e.g., \cite{S. Dong,Z. Li,X. Li-Zhou-2002,X. Li-Zhou-Ait Rami-2003,Z. Gajic,I.G. Ivanov,A. G. Wu}). To illustrate, Li et al. \cite{X. Li-Zhou-Ait Rami-2003} established 
the equivalence between the well-posedness of indefinite stochastic LQ problems and the solvability of generalized CAREs. Recognizing the inherent nonlinearity and coupling of these equations, Gajic and Borno \cite{Z. Gajic} developed an order-reduction 
iterative scheme that decouples algebraic Lyapunov equations, guaranteeing convergence of the solution sequence to the CARE solution. Subsequently, Ivanov \cite{I.G. Ivanov} improved the scheme of \cite{Z. Gajic} 
by addressing its reliance solely on information from the previous iteration, and proposed an accelerated algorithm that incorporates the most recently updated information within the same iteration. Building on both methods, Wu et al. \cite{A. G. Wu} further 
advanced the state of the art by introducing an iterative update that simultaneously exploits both previous-step and current-step information. In summary, the vast majority of existing results for LQ problems with regime switching remain firmly rooted in CARE-based solutions, 
which inherently demand full knowledge of the system dynamics. Yet, in many practical environments, precise dynamic models are either unavailable or prohibitively expensive to obtain, rendering such model-dependent approaches infeasible.

Fortunately, reinforcement learning (RL) is capable of addressing LQ problems when information about the system dynamics is not readily available. Among RL approaches, 
Q-learning is a superior method that does not require any prior knowledge of the system dynamics. In recent years, Q-learning methods have been widely applied to solving LQ problems for both 
discrete-time \cite{K. Du,Syed Ali Asad Rizvi,Tao Wang} and continuous-time systems \cite{J. Y. Lee,M. Palanisamy,Corrado Possieri,W. J. Yang}. On the one hand, in discrete-time cases, Rizvi and Lin \cite{Syed Ali Asad Rizvi} proposed an output feedback-based Q-learning algorithm, 
which enables the design of LQ controllers directly from input-output data and avoids excitation noise bias. For the situation with a single random parameter sample at each time step, 
Du et al. \cite{K. Du} adopted the Q-learning approach and proved both convergence and stability of the algorithm. For stochastic systems, Wang et al. \cite{Tao Wang} transformed the stochastic problem into deterministic 
one and employed a Q-learning method to solve LQ problems. On the other hand, in continuous-time cases, Palanisamy et al. \cite{M. Palanisamy} solved discounted LQ control 
problems for deterministic systems by employing a parameterized Q-function, while Yang et al. \cite{W. J. Yang} focused on decentralized LQ control for stochastic systems using a decentralized Q-function. 
Possieri and Sassano \cite{Corrado Possieri} employed the Kleinman algorithm in a data-driven manner and proposed an Actor/Critic type Q-learning algorithm that does not require the existence of 
a known initial stabilizing feedback policy. Furthermore, Jia and Zhou \cite{Y. W. Jia} investigated q-learning under an entropy-regularized exploratory diffusion process framework in a stochastic case. 
Although Q-learning methods have made progress in solving LQ problems with the systems discussed above, research on
employing Q-learning methods to solve LQ problems with regime switching remains relatively scarce. Therefore, further research into the theoretical analysis and algorithmic design of Q-learning
for stochastic LQ optimal control problems with regime switching, particularly in complex scenarios where information about the system dynamics is unavailable, is of theoretical
and practical importance.

Inspired by the aforementioned works, this paper employs a Q-learning method to address infinite-horizon continuous-time stochastic LQ optimal control problems with regime switching, 
where the drift and diffusion terms in the system dynamics depend on both the state and control. Two proposed algorithms, including on-policy and off-policy Q-learning algorithms, 
reduce the reliance on information about the system dynamics.
Furthermore, we systematically prove the equivalence, stability, and convergence of two algorithms. In addition, the effectiveness and feasibility 
of the developed algorithms are demonstrated through numerical experiments. The main contributions of this paper are summarized as follows:
\begin{enumerate}[label=(\arabic*)]
\item For infinite-horizon continuous-time stochastic LQ optimal control problems with regime switching, we propose on-policy
and off-policy Q-learning algorithms. Unlike the work of Wu et al.\cite{A. G. Wu}, the proposed Q-learning algorithms only require state trajectory data, and do not require any prior knowledge of
the system dynamics.
\item Due to the uncertainty of Markov chains at next moment, it is impossible to know which Markov chain state will be in at moment $t+\Delta t$ ($t \geq 0$, $\Delta t >0$). 
For this reason, we cannot identify which $P_j, j=1,2,...,\mathbf{K}$ will be chosen at the future moment $t+\Delta t$. Meanwhile, all of the $\{P_j\}_{j=1}^{\mathbf{K}}$ are coupled together.
By skillfully using system stability properties, we solve the coupled problem over the entire time interval $[t,+\infty)$, thereby significantly reducing computational complexity and decoupling $\{P_j\}_{j=1}^{\mathbf{K}}$.
\item The stability and convergence of Lyapunov Iteration Scheme are first proved. Based on the equivalence between Q-learning algorithms and Lyapunov Iteration Scheme, we obtain the stability and convergence of 
the proposed on-policy and off-policy Q-learning algorithms. Compared to the coupling of $\{P_j\}_{j=1}^{ \mathbf{K} }$ in Lyapunov Iteration Scheme, Q-learning algorithms can directly solve $\{P_j\}_{j=1}^{ \mathbf{K} }$ 
without requiring the system coefficients.

\end{enumerate}

The remainder of this paper is organized as follows: Section \ref{sec.2} introduces the notation and formulates the system model and the LQ problem with regime switching. Section \ref{sec.3} presents Lyapunov Iteration Scheme, on-policy and off-policy Q-learning algorithms, 
and proves the equivalence, stability and convergence of the proposed algorithms. Section \ref{sec.4} provides a numerical example to verify the effectiveness of the proposed algorithms. Finally, Section \ref{sec.5} concludes the paper.

\section{Notation and Formulation}\label{sec.2}
~~~~Let $\mathbb R^n$ denote the $n$-dimensional Euclidean space with Euclidean norm $|\cdot|$. 
Let $\mathbb R^{n\times r}$ be the set of all $(n\times r)$ real matrices. 
$M^{\top}$ denotes the transpose of a vector or matrix $M$. 
$I_n\in \mathbb R^{n\times n}$ denotes the $n$-dimensional identity matrix.
We use $\mathcal S^n$, $\mathcal S^n_+$ and $\mathcal S^n_{++}$ to denote the collections of all symmetric matrices, positive semidefinite matrices, and positive definite matrices in $\mathbb R^{n\times n}$, respectively.
Moreover, if a matrix $M\in\mathcal S^n_{++}$(resp. $M\in\mathcal S^n_{+}$) is positive definite (resp. positive semidefinite), we usually write $M>0$ (resp. $M\geq 0$).
Let $(\Omega, \mathcal F,\mathbb F, {\mathbb P} )$ be a given filtered probability space on which there exists a standard one-dimensional Brownian motion $W(s)$ on $[0,\infty)$, 
and a continuous-time stationary Markov chain $\alpha_s$ taking values in a finite state space $\mathcal{M}= \{1, 2, \cdots, \mathbf{K} \}$, where $ \mathbf{K} \in\mathbb{N}^{+}$ and $ \mathbf{K} > 1$ with the generator $\Pi=(\pi_{kj})_{k,j\in \mathcal{M}}$. 
Assume that $W$ and $\alpha$ are independent.
For $s\geq 0$, we define $\mathcal F_s = \sigma\{W(\tau), \alpha_\tau : 0 \leq \tau \leq s\}$ augmented by all \(\mathbb P\)-null sets, and denote $\mathbb F=\{\mathcal F_s\}_{s\ge0}$.
For $t\geq 0$, we define \(L_{\mathbb F}^{2}(t,\infty;\mathbb R^n)\) as the Hilbert space of all \(\mathbb R^n\)-valued \(\mathbb{F}\)-progressively measurable processes \(\varphi(\cdot)\)
with the finite norm $\|\varphi(\cdot)\|=\left[\mathbb E\int_{t}^\infty |\varphi(s)|^2 ds\right]^{\frac{1}{2}}<\infty$. 
Furthermore, we use $\otimes$ to denote the Kronecker product. For any matrix $P \in \mathcal{S}^n$, we define 
\begin{align*}
vec(P)&=[p_{11}, p_{21}, \cdots, p_{n1}, p_{12}, p_{22}, \cdots, p_{n-1,n}, p_{nn}]^\top \in \mathbb{R}^{n^2}, \\
vech(P)&=[p_{11}, 2p_{12}, \cdots, 2p_{1n}, p_{22}, 2p_{23}, \cdots, 2p_{n-1,n}, p_{nn}]^\top \in \mathbb{R}^{\frac{1}{2}n(n+1)},
\end{align*}
where $p_{ij}~(i,j=1,2,3,...)$ is the $(i,j)$th element of $P$.
According to \cite{J. J. Murray}, there exists a matrix $\mathcal{T}\in\mathbb{R}^{n^2 \times \frac{1}{2}n(n+1)}$ with $\mathrm{rank}(\mathcal{T})=\frac{1}{2}n(n+1)$ such that $vec(P)=\mathcal{T}\,vech(P)$ for any $P\in \mathcal{S}^n$. Correspondingly, for any $H \in \mathbb{R}^{n\times r}$, we define
\[
\bar{H} := \left( H^\top \otimes H^\top \right) \mathcal{T} \in \mathbb{R}^{r^2 \times \frac{1}{2}n(n+1)}, 
\]
which satisfies $(H^\top \otimes H^\top) \,vec(P) = \bar{H}\,vech(P)$. For a vector $x\in\mathbb{R}^n$, we define
\[
\bar{x} = 
\begin{bmatrix}
x_1^2,\,x_1x_2,\,\ldots,\,x_1x_n,\,x_2^2,\,x_2x_3,\,\ldots,\,x_{n-1}x_n,\,x_n^2
\end{bmatrix}^\top \in \mathbb{R}^{\frac{1}{2}n(n+1)},
\]
where $x_{i}~(i=1,2,3,...)$ is the $i$th element of $x$.

In this paper, we study a linear stochastic differential equation with regime switching 
\begin{equation}\label{system}
	\left\{
	\begin{aligned}
		dX(s) & = [A(\alpha_s)X(s)+B(\alpha_s)u(s)]ds+[C(\alpha_s)X(s)+D(\alpha_s)u(s)]dW(s),~s\in[t,\infty), \\
X(t) & = x \in {\mathbb R}^n,~~~\alpha_t=k, 
	\end{aligned}
	\right.
\end{equation}
where $A(\alpha_s) = A_k$, $B(\alpha_s) = B_k$, $C(\alpha_s) = C_k$, $D(\alpha_s) = D_k$ when $\alpha_s = k ~(k\in \mathcal{M})$. Here, the matrices $A_k$, etc. are given with appropriate dimensions. 
The Markov chain $\alpha_s$ has transition probabilities
\begin{eqnarray}\label{eq:pi}
{\bf P}\{\alpha_{s + \Delta s} = j \;\vert\; \alpha_s = k\} = \left\{
\begin{array}{ll}
\pi_{kj}\Delta s + o(\Delta s), & \mbox{if } k \not = j, \\
\vspace{-0.4cm} \\ 
1 + \pi_{kk}\Delta s + o(\Delta s), & \mbox{if } k = j,
\end{array}\right.
\end{eqnarray}
where $\pi_{kj} \geq 0$ for $k \not = j$ and $\pi_{kk} = - \sum_{j \not = k}\pi_{kj}\leq 0$.

The cost functional is given by
\begin{equation}\label{cost}
\begin{aligned}
&~~~J(t,x,k;u(\cdot))
\\&=\mathbb E\Big\{\int_{t}^\infty\Big[X(s)^\top N(\alpha_s)X(s)+2u(s)^\top S(\alpha_s) X(s)+u(s)^\top R(\alpha_s)u(s)\Big]ds\;\Big|\;X(t)=x,\alpha_t=k\Big\},
\end{aligned}
\end{equation}
where 
$N(\alpha_s) = N_k$, $R(\alpha_s) = R_k$ and $S(\alpha_s) = S_k$ when $\alpha_s = k ~(k\in \mathcal{M})$ with appropriate dimensions. 
From now on, we assume that the weighting matrices satisfy the positive definite condition: $R_k>0$ and $N_k-S_k^\top R_k^{-1}S_k > 0$ for any $k\in \mathcal{M}$.

\begin{definition}
\label{mean-stablity}
A control $u(\cdot)$ is called mean-square stabilizing
with respect to an initial $(t,x,k)$ if the corresponding state $X(\cdot)$ in system \eqref{system} satisfies 
\[\lim_{s \rightarrow \infty} \mathbb E[X(s)^\top X(s)]=0,\] and system \eqref{system} is called mean-square stabilizable. 
Moreover, if the feedback control 
$u(s) =\sum_{k=1}^ \mathbf{K} K_kX(s)\chi_{\{\alpha_s=k\}}(s)$ makes system \eqref{system} stabilize, then $ (K_1, K_2, \cdots, K_{ \mathbf{K} })$ is called a stabilizer of system \eqref{system}.
\end{definition}

When system \eqref{system} is mean-square stabilizable for a given $(t,x,k)\in[0,\infty)\times\mathbb R^n\times\mathcal{M}$, we define the corresponding set of admissible controls as
\begin{equation}
\mathcal U_{ad}(t,x,k)=\{u(\cdot)\in L_{\mathbb F}^2(t,\infty;\mathbb R^r): u(\cdot){\rm ~is~stabilizing~with~respect~to~ }(t,x,k)\}.
\end{equation}

In this paper, we focus on the following problem, where Problem (LQ-RS) is the abbreviation of the LQ problem with regime switching. \\

\noindent{\bf Problem (LQ-RS).} For a given initial $(t,x,k)\in[0,\infty)\times\mathbb R^n\times\mathcal{M}$, find a control $u^*(\cdot)\in\mathcal U_{ad}(t,x,k)$ such that
\[
J(t,x,k;u^*(\cdot))=\inf_{u(\cdot)\in \mathcal U_{ad}(t,x,k)}J(t,x,k;u(\cdot))\triangleq V(t,x,k),
\]
where $V(t,x,k)$ is called the value function of Problem (LQ-RS). \bigskip

For a given $(t,x,k)\in[0,\infty)\times\mathbb R^n\times\mathcal{M}$, Problem (LQ-RS) is called well-posed if $V(t,x,k)>-\infty$.
A well-posed problem is called {\it attainable} at $(t,x,k)$ if there exists a control $u^*(\cdot)$ such that $J(t,x,k;u^*(\cdot))=V(t,x,k)$. 
In this case, $u^*(\cdot)$ is called an {\it optimal control}, and $X^*(\cdot) \equiv X\big(\cdot;t,x,k, u^*(\cdot)\big)$ is called an
{\it optimal state} corresponding to $u^*(\cdot)$. 

\begin{assumption}\label{ass1}
	System \eqref{system} is mean-square stabilizable. 
\end{assumption}

Next, we list three lemmas that are important in our subsequent analysis.
First, we present the generalized Itô's formula from \cite[Theorem~70]{P. E. Protter}.
\begin{lemma}
\label{lem:Ito} 
Let $b(s, X, k)$ and $\sigma(s, X, k)$ be given ${\mathbb R}^n$-valued, 
$\mathcal F_s$-adapted process, $k\in \mathcal{M}$, and
\begin{eqnarray}dX(s) = b(s, X(s), \alpha_s)ds + \sigma(s, X(s), \alpha_s)dW(s).
\nonumber
\end{eqnarray} 
If $\varphi(\cdot, \cdot, k) \in C^{1,2}([0, \infty) \times {\mathbb R}^n)$ 
for all $k\in \mathcal{M}$, then we have
\begin{equation*}
	d\varphi(s, X(s), \alpha_s)= \mathcal{L}\varphi(s, X(s), \alpha_s)ds+\varphi_x^\top(s, X(s), \alpha_s)\sigma(s, X(s), \alpha_s)dW(s),
\end{equation*}
where
\[
\begin{aligned}
\mathcal{L}\varphi(s, x, k)
:=&~ \varphi_t(s, x, k) + \varphi_x(s, x,k)^\top b(s, x, k) \\ 
&~ + \frac{1}{2}{\rm tr}[\sigma(s, x, k)^\top \varphi_{xx}(s, x, k)\sigma(s, x, k)] + \sum_{j=1}^{ \mathbf{K} }\pi_{kj}\varphi(s, x, j). 
\end{aligned}
\] 
\end{lemma}

The following lemma illustrates that the stabilizability of system \eqref{system} is related to 
the feasibility of certain coupled Lyapunov inequalities, and not dependent on initial $(t,x,k)$. Please refer to \cite[Theorem 3.1]{L. E. Ghaoui} or \cite[Lemma 2.5]{X. Li-Zhou-Ait Rami-2003}. 

\begin{lemma}\label{lemma-stabilizer}
	System \eqref{system} is mean-square stabilizable if and only if there exists matrices $K_1, K_2, \cdots,$ $K_{ \mathbf{K} }$ and symmetric matrices
$P_1,P_2, \cdots, P_{ \mathbf{K} }\in \mathcal{S}^n_{++}$ such that
\begin{equation}\label{cond1}
\begin{aligned} 
(A_k + B_kK_k)^\top P_k + P_k(A_k + B_kK_k) \\
+ (C_k + D_kK_k)^\top P_k(C_k + D_kK_k) & + \sum_{j=1}^{ \mathbf{K} }\pi_{kj}P_j < 0, ~~ k\in \mathcal{M}.
\end{aligned}
\end{equation} 
In this case, for any $\Lambda_1,\Lambda_2, \cdots, \Lambda_{ \mathbf{K} } \in \mathcal{S}^n_{+}$ $(resp., \mathcal{S}^n_{++})$, the Lyapunov eqautions
\begin{equation}\label{cond3}
\begin{aligned}
(A_k + B_kK_k)^\top P_k + P_k(A_k + B_kK_k) \\
+ (C_k + D_kK_k)^\top P_k(C_k + D_kK_k) & + \sum_{j=1}^{ \mathbf{K} }\pi_{kj}P_j + \Lambda_k = 0 
\end{aligned} 
\end{equation}
with $k\in\mathcal{M}$ admit a unique solution $(P_1,P_2 ,\cdots, P_{ \mathbf{K} }) \in (\mathcal S^n_{+})^{ \mathbf{K} }$ $(resp., (\mathcal S^n_{++})^{ \mathbf{K} })$.
\end{lemma} 

Referring to Theorems 4.4 and 5.2 in \cite{X. Li-Zhou-Ait Rami-2003}, we present the lemma for optimal control of Problem (LQ-RS) by coupled generalized algebraic 
Riccati equations (CGAREs). 	
\begin{lemma} 
If there exists a solution $(P_1, P_2, \cdots, P_{ \mathbf{K} })\in(\mathcal S^n_{++})^{ \mathbf{K} }$ to the following CGAREs
\begin{equation}\label{SARE}
\begin{aligned}
&A_k^\top P_k+P_kA_k+C_k^\top P_kC_k+N_k+\sum_{j=1}^{ \mathbf{K} }\pi_{kj}P_j\\&~~~ -(P_kB_k+C_k^\top P_kD_k+S_k^\top)(R_k+D_k^\top P_k D_k)^{-1}(B_k^\top P_k+D_k^\top P_kC_k+S_k )=0,~~k\in \mathcal{M},\\
\end{aligned}
\end{equation}
then the optimal control for Problem (LQ-RS) is 
\begin{eqnarray}\label{eq:Gfeed} 
\begin{array}{l} 
u^*(s) = K(\alpha_s)X^*(s)=\sum_{k=1}^ \mathbf{K} K_kX^*(s)\chi_{\{\alpha_s=k\}}(s)\end{array} 
\end{eqnarray}
with $K_k=-(R_k + D_k^\top P_kD_k)^{-1}(B_k^\top P_k + D_k^\top P_kC_k + S_k)$, 
where $X^*$ is the corresponding optimal state with $u^*$ of system \eqref{system}. The value function is
\begin{equation}
\begin{aligned}
	&~~~V(t,x,k)=x^\top P_kx \\&=\mathbb E\Big\{\int_{t}^\infty X^{*}(s)^{\top}\Big[ N(\alpha_s)+2K(\alpha_s)^{\top} S(\alpha_s)+K(\alpha_s)^{\top} R(\alpha_s)K(\alpha_s)\Big]X^{*}(s)ds\;\Big|\;X^{*}(t)=x,\alpha_t=k\Big\},	\end{aligned}
\end{equation}
for $k\in \mathcal{M}.$
\end{lemma}

\section{Q-learning method for Problem (LQ-RS)}\label{sec.3}
~~~~In this section, we present the main results of this paper. 
We first introduce Lyapunov Iteration Scheme for Problem (LQ-RS) and 
prove its stability and convergence.
Based on this scheme, we further develop on-policy and off-policy Q-learning algorithms, 
which can decouple $\{P_j\}_{j=1}^{ \mathbf{K} }$ without the system coefficients.
\subsection{Lyapunov Iteration Scheme}
~~~~~Now, we present {\bf Lyapunov Iteration Scheme} as follows:
\begin{enumerate}
\item[(1)] {\bf Initialization:} Select any given stabilizer $(K_1^{(0)},K_2^{(0)},\cdots, K_{ \mathbf{K} }^{(0)})$ for system \eqref{system}.
\item[(2)] {\bf Lyapunov Recursion:} At each iteration step $i ~(i=0,1,2,...)$, solve $P_k^{(i+1)}~(k \in \mathcal{M})$ from the equations 
\begin{equation}\label{Lyapunov Recursion}
\begin{aligned}
&(A_k+B_kK_k^{(i)})^\top P_k^{(i+1)}+P_k^{(i+1)}(A_k+B_kK_k^{(i)})+(C_k+D_kK_k^{(i)})^\top P_k^{(i+1)}(C_k+D_kK_k^{(i)})\\
&~~~~\qquad\qquad+\sum_{j=1}^{ \mathbf{K} }\pi_{kj}P_j^{(i+1)}+K_k^{(i)\top} R_kK_k^{(i)}+S_k^\top K_k^{(i)}+K_k^{(i)\top} S_k+N_k=0,~~k\in \mathcal{M}.
\end{aligned}
\end{equation}
\item[(3)] {\bf Policy Improvement:} Update $K_k^{(i+1)}~(k \in \mathcal{M})$ by
\begin{equation}\label{algorithm 1-improvement}
\begin{aligned}
K_k^{(i+1)}&=-(R_k+D_k^\top P_k^{(i+1)} D_k)^{-1}(B_k^\top P_k^{(i+1)}+D_k^\top P_k^{(i+1)}C_k+S_k ),~~k\in \mathcal{M}.
\end{aligned} 
\end{equation}
\end{enumerate} 

In Lyapunov Iteration Scheme, 
$\{(K_1^{(i)},K_2^{(i)},\cdots, K_{ \mathbf{K} }^{(i)})\}_{i=0}^\infty$ should be the stabilizers of system \eqref{system} at each iteration, 
i.e., it is necessary to require that $(K_1^{(i)},K_2^{(i)},\cdots, K_{ \mathbf{K} }^{(i)})$ is stepwise stable. 
Moreover, $\{(P_1^{(i)},P_2^{(i)},\cdots,P_{ \mathbf{K} }^{(i)})\}_{i=1}^{\infty}$ should be convergent when $i\rightarrow \infty$. The following results confirm these. 

\begin{theorem}\label{theorem 1}
Assume Assumption \ref{ass1} holds. Let $(K_1^{(0)},K_2^{(0)},\cdots, K_{ \mathbf{K} }^{(0)})$ be the known initial stabilizer for system \eqref{system}, then $\{(K_1^{(i)},K_2^{(i)},\cdots, K_{ \mathbf{K} }^{(i)})\}_{i=1}^\infty$ updated by \eqref{algorithm 1-improvement} are stabilizers. 
Moreover, the solution $P_k^{(i+1)}\in \mathcal S_{++}^n$ $(k\in \mathcal{M})$ of Lyapunov Recursion \eqref{Lyapunov Recursion} is unique at each iteration step $i$. 
\end{theorem}
\begin{proof}
Since $K^{(0)}_k$ is a stabilizer for system \eqref{system} and
\begin{align*}
&\quad K_k^{(0)\top} R_kK_k^{(0)}+S_k^\top K_k^{(0)}+K_k^{(0)\top} S_k+N_k\\
&=N_k-S_k^\top R_k^{-1}S_k+(R_kK_k^{(0)}+S_k)^\top R_k^{-1}(R_kK_k^{(0)}+S_k)>0, ~~k\in \mathcal{M},
\end{align*}
from Lemma \ref{lemma-stabilizer}, there exists a unique solution $P^{(1)}_k\in\mathcal S^n_{++}$ of Lyapunov Recursion \eqref{Lyapunov Recursion} with $i=0$.

Now, we prove that for each $i\geq 1$ and every $k\in \mathcal{M}$, if $K_k^{(i-1)}$ is a stabilizer and $P_k^{(i)}\in\mathcal S^n_{++}$ is the unique solution of Lyapunov Recursion \eqref{Lyapunov Recursion}, then $K_k^{(i)}=-(R_k+D_k^\top P_k^{(i)} D_k)^{-1}(B_k^\top P_k^{(i)}+D_k^\top P_k^{(i)}C_k+S_k )$ is also a stabilizer and $P_k^{(i+1)}\in\mathcal S^n_{++}$ for $k\in \mathcal{M}$. Indeed, a simple calculation shows that 
\begin{align}\label{stabilizer}
&\quad~(A_k+B_kK_k^{(i)})^\top P_k^{(i)}+P_k^{(i)}(A_k+B_kK_k^{(i)})+(C_k+D_kK_k^{(i)})^\top P_k^{(i)}(C_k+D_kK_k^{(i)})+\sum_{j=1}^{ \mathbf{K} }\pi_{kj}P_j^{(i)} \nonumber\\
&=(A_k+B_kK_k^{(i)})^\top P_k^{(i)}+P_k^{(i)}(A_k+B_kK_k^{(i)})+(C_k+D_kK_k^{(i)})^\top P_k^{(i)}(C_k+D_kK_k^{(i)})+\sum_{j=1}^{ \mathbf{K} }\pi_{kj}P_j^{(i)}\nonumber\\
&\quad-\Big[(A_k+B_kK_k^{(i-1)})^\top P_k^{(i)}+P_k^{(i)}(A_k+B_kK_k^{(i-1)})+(C_k+D_kK_k^{(i-1)})^\top P_k^{(i)}(C_k+D_kK_k^{(i-1)})\nonumber\\
&\quad\quad+\sum_{j=1}^{ \mathbf{K} }\pi_{kj}P^{(i)}_j+K_k^{(i-1)\top} R_kK_k^{(i-1)}+S_k^\top K_k^{(i-1)}+K_k^{(i-1)\top} S_k+N_k\Big]\nonumber\\
&=-\big[K_k^{(i-1)\top} R_kK_k^{(i-1)}+S_k^\top K_k^{(i-1)}+K_k^{(i-1)\top} S_k+N_k\big]\nonumber\\
&\quad-(K_k^{(i-1)}-K_k^{(i)})^\top D_k^\top P_k^{(i)}D_k(K_k^{(i-1)}-K_k^{(i)})\nonumber\\
&\quad-\Big[(K_k^{(i-1)}-K_k^{(i)})^\top \big[B_k^\top P_k^{(i)}+D_k^\top P_k^{(i)}(C_k+D_kK_k^{(i)})\big]\nonumber\\
&\quad+\big[P_k^{(i)}B_k+(C_k+D_kK_k^{(i)})^\top P_k^{(i)}D_k\big](K_k^{(i-1)}-K_k^{(i)})\Big],~~k\in \mathcal{M}.
\end{align}
From \eqref{algorithm 1-improvement}, we have 
\begin{equation}\label{K-i+1}
B_k^\top P_k^{(i)}+D_k^\top P_k^{(i)}C_k=-(R_k+D_k^\top P_k^{(i)}D_k)K_k^{(i)}-S_k,~~k\in \mathcal{M}.
\end{equation}
Plugging \eqref{K-i+1} into \eqref{stabilizer}, since $N_k-S_k^\top R_k^{-1}S_k>0$, one gets
\begin{align}\label{stabilizer2}
&\quad(A_k+B_kK_k^{(i)})^\top P_k^{(i)}+P_k^{(i)}(A_k+B_kK_k^{(i)})+(C_k+D_kK_k^{(i)})^\top P_k^{(i)}(C_k+D_kK_k^{(i)})+\sum_{j=1}^{ \mathbf{K} }\pi_{kj}P_j^{(i)}\nonumber\\
&=-\big[K_k^{(i)\top} R_k K_k^{(i)}+S_k^\top K_k^{(i)}+K_k^{(i)\top} S_k+N_k\big]\nonumber\\
&\quad-(K_k^{(i-1)}-K_k^{(i)})^\top(R_k+ D_k^\top P_k^{(i)}D_k)(K_k^{(i-1)}-K_k^{(i)})\nonumber\\
&=-\big[N_k-S_k^\top R_k^{-1}S_k+(R_k K_k^{(i)}+S_k)^\top R_k^{-1} (R_kK_k^{(i)}+S_k)\big]\nonumber\\
&\quad-(K_k^{(i-1)}-K_k^{(i)})^\top (R_k+D_k^\top P_k^{(i)}D_k)(K_k^{(i-1)}-K_k^{(i)})\nonumber\\
&<0,\qquad\qquad\qquad\qquad\qquad k\in \mathcal{M},
\end{align} 
so $(K_1^{(i)},K_2^{(i)},\cdots,K_{ \mathbf{K} }^{(i)})$ is a stabilizer by Lemma \ref{lemma-stabilizer}. Moreover, Lyapunov Recursion \eqref{Lyapunov Recursion}
admits a unique solution $P_k^{(i+1)}\in\mathcal S^n_{++}$ as 
$K_k^{(i)\top} R_kK_k^{(i)}+S_k^\top K_k^{(i)}+K_k^{(i)\top} S_k+N_k>0,~k\in \mathcal{M}$. 
We obtain the conclusion.
\end{proof} 

\begin{theorem}\label{theorem 2}
The iteration $\{(P_1^{(i)},P_2^{(i)},\cdots,P_{ \mathbf{K} }^{(i)})\}_{i=1}^{\infty}$ of Lyapunov Recursion \eqref{Lyapunov Recursion} converges to the solution $(P_1,P_2,\cdots,P_{ \mathbf{K} })\in(\mathcal S^n_{++})^{ \mathbf{K} }$ of CGAREs \eqref{SARE}.
\end{theorem}
\begin{proof}
Note $P_k^{(i+1)}$ satisfies 
Lyapunov Recursion \eqref{Lyapunov Recursion}.
Denote
\[
\Delta P_k^{(i+1)}=P_k^{(i)}-P_k^{(i+1)},~~~~\Delta K_k^{(i)}=K_k^{(i-1)}-K_k^{(i)},~~~k\in \mathcal{M},
\]
for $i=1,2,\cdots$, and we have
\begin{equation}\label{Pi-Pi+1}
\begin{aligned}
0&=(A_k+B_kK_k^{(i-1)})^\top P_k^{(i)}+P_k^{(i)}(A_k+B_kK_k^{(i-1)})+(C_k+D_kK_k^{(i-1)})^\top P_k^{(i)}(C_k+D_kK_k^{(i-1)})\\
&\quad+\sum_{j=1}^{ \mathbf{K} }\pi_{kj}P_j^{(i)}+K_k^{(i-1)\top} R_kK_k^{(i-1)}+S_k^\top K_k^{(i-1)}+K_k^{(i-1)\top}S_k \\
&\quad-\big[(A_k+B_kK_k^{(i)})^\top P_k^{(i+1)}+P_k^{(i+1)}(A_k+B_kK_k^{(i)})+(C_k+D_kK_k^{(i)})^\top P_k^{(i+1)}(C_k+D_kK_k^{(i)})\\
&\quad\quad+\sum_{j=1}^{ \mathbf{K} }\pi_{kj}P_j^{(i+1)}+K_k^{(i)\top} R_kK_k^{(i)}+S_k^\top K_k^{(i)}+K_k^{(i)\top}S_k\big] \\
&=(A_k+B_kK_k^{(i)})^\top\Delta P_k^{(i+1)}+\Delta P_k^{(i+1)}(A_k+B_kK_k^{(i)})+(C_k+D_kK_k^{(i)})^\top \Delta P_k^{(i+1)}(C_k+D_kK_k^{(i)})\\
&\quad+\sum_{j=1}^{ \mathbf{K} }\pi_{kj}\Delta P_j^{(i+1)}+\Delta K_k^{(i)\top} B_k^\top P_k^{(i)}+P_k^{(i)}B_k\Delta K_k^{(i)}\\
&\quad+(C_k+D_kK_k^{(i-1)})^\top P_k^{(i)}(C_k+D_kK_k^{(i-1)})-(C_k+D_kK_k^{(i)})^\top P_k^{(i)}(C_k+D_kK_k^{(i)})\\
&\quad+K_k^{(i-1)\top} R_kK_k^{(i-1)}-K_k^{(i)\top} R_kK_k^{(i)}+S_k^\top\Delta K_k^{(i)}+\Delta K_k^{(i)\top} S_k,~~~~~~~~~k\in \mathcal{M}.
\end{aligned}
\end{equation}
From \eqref{algorithm 1-improvement}, one gets 
\begin{equation}\label{C}
\begin{aligned}
&\quad(C_k+D_kK_k^{(i-1)})^\top P_k^{(i)}(C_k+D_kK_k^{(i-1)})-(C_k+D_kK_k^{(i)})^\top P_k^{(i)}(C_k+D_kK_k^{(i)})\\
&=\Delta K_k^{(i)\top} D_k^\top P_k^{(i)}D_k\Delta K_k^{(i)}+\Delta K_k^{(i)\top} D_k^\top P_k^{(i)}(C_k+D_kK_k^{(i)})\\
&\quad+(C_k+D_kK_k^{(i)})^\top P_k^{(i)}D_k\Delta K_k^{(i)},\qquad\qquad~~~k\in \mathcal{M},
\end{aligned}
\end{equation}
and 
\begin{equation}\label{Ki}
\begin{aligned} 
&\quad K_k^{(i-1)\top} R_kK_k^{(i-1)}-K_k^{(i)\top} R_kK_k^{(i)}\\
&=\Delta K_k^{(i)\top} R_k\Delta K_k^{(i)}+\Delta K_k^{(i)\top} R_kK_k^{(i)}+K_k^{(i)\top} R_k\Delta K_k^{(i)},~~~k\in \mathcal{M}.
\end{aligned}
\end{equation}
Combining \eqref{Pi-Pi+1}-\eqref{Ki}, we deduce 
\begin{equation}\label{eq:D}
\begin{aligned}
&(A_k+B_kK_k^{(i)})^\top\Delta P_k^{(i+1)}+\Delta P_k^{(i+1)}(A_k+B_kK_k^{(i)})+(C_k+D_kK_k^{(i)})^\top \Delta P_k^{(i+1)}(C_k+D_kK_k^{(i)})\\
&\quad+\sum_{j=1}^{ \mathbf{K} }\pi_{kj}\Delta P_j^{(i+1)}+\Delta K_k^{(i)\top}(R_k+D_k^\top P_k^{(i)}D_k)\Delta K_k^{(i)}\\
&\quad+\Delta K_k^{(i)\top}\big[B_k^\top P_k^{(i)}+D_k^\top P_k^{(i)}C_k+S_k+(R_k+D_k^\top P_k^{(i)}D_k)K_k^{(i)}\big]\\
&\quad+\big[B_k^\top P_k^{(i)}+D_k^\top P_k^{(i)}C_k+S_k+(R_k+D_k^\top P_k^{(i)}D_k)K_k^{(i)}\big]^\top \Delta K_k^{(i)}=0,~~~k\in \mathcal{M}.
\end{aligned}
\end{equation}
By \eqref{algorithm 1-improvement}, we have
\begin{align}\label{eq:RDPD}
-(R_k+D_k^\top P_k^{(i)}D_k)K_k^{(i)}=B_k^\top P_k^{(i)}+D_k^\top P_k^{(i)}C_k+S_k,~~~k\in \mathcal{M}.
\end{align}
Plugging \eqref{eq:RDPD} into \eqref{eq:D}, one gets
\begin{equation}\label{Delta}
\begin{aligned}
&(A_k+B_kK_k^{(i)})^\top\Delta P_k^{(i+1)}+\Delta P_k^{(i+1)}(A_k+B_kK_k^{(i)})+(C_k+D_kK_k^{(i)})^\top \Delta P_k^{(i+1)}(C_k+D_kK_k^{(i)})\\
&~~~~~~~~\qquad\qquad+\sum_{j=1}^{ \mathbf{K} }\pi_{kj}\Delta P_j^{(i+1)}+\Delta K_k^{(i)\top}(R_k+D_k^\top P_k^{(i)}D_k)\Delta K_k^{(i)}=0,~~~k\in \mathcal{M}.
\end{aligned}
\end{equation}
Since $(K_1^{(i)},K_2^{(i)},\cdots,K_{ \mathbf{K} }^{(i)})$ is a stabilizer of system \eqref{system} and 
$$\Delta K_k^{(i)\top}(R_k+D_k^\top P_k^{(i)}D_k)\Delta K_k^{(i)}\geq 0, ~~~k\in \mathcal{M},$$ 
from Lemma \ref{lemma-stabilizer}, Lyapunov equation \eqref{Delta} admits a unique solution $\Delta P_k^{(i+1)}\geq 0$. 
Then, $\{(P_1^{(i)},P_2^{(i)},\cdots,P_{ \mathbf{K} }^{(i)})\}_{i=1}^{\infty}$ is monotonically nonincreasing. 
Since $P_k^{(i)}> 0$, and $P_k^{(i)}$ has a low bound, so $\{(P_1^{(i)},P_2^{(i)},\cdots,$ $P_{ \mathbf{K} }^{(i)})\}_{i=1}^{\infty}$ converges to $(P_1,P_2,\cdots,P_{ \mathbf{K} })\in(\mathcal S^n_{+})^{ \mathbf{K} }$. 

Next, we prove that $P_k$ is the solution of CGAREs \eqref{SARE}.
When $i\rightarrow\infty$, as $P_k^{(i)}\rightarrow P_k$, we have
$\{K_k^{(i)}\}_{i=1}^{\infty}$ converges to 
\begin{equation}\label{eq:Km}
K_k=-(R_k+D_k^\top P_k D_k)^{-1}(B_k^\top P_k+D_k^\top P_kC_k+S_k ),~~~k\in \mathcal{M}.
\end{equation}
Moreover, $(P_1,P_2,\cdots,P_{ \mathbf{K} })$ satisfies 
\begin{equation}\label{Lyapunov K}
\begin{aligned}
&(A_k+B_kK_k)^\top P_k+P_k(A_k+B_kK_k)+(C_k+D_kK_k)^\top P_k(C_k+D_kK_k)\\
&~~\qquad\qquad\qquad+\sum_{j=1}^{ \mathbf{K} }\pi_{kj}P_j+K_k^\top R_kK_k+S_k^\top K_k+K_k^\top S_k+N_k=0,~~~k\in \mathcal{M}.
\end{aligned}
\end{equation}
Since 
$K_k^\top R_kK_k+S_k^\top K_k+K_k^\top S_k+N_k> 0$,
one gets
\[
(A_k+B_kK_k)^\top P_k+P_k(A_k+B_kK_k)+(C_k+D_kK_k)^\top P_k(C_k+D_kK_k)+\sum_{j=1}^{ \mathbf{K} }\pi_{kj}P_j<0,~~~k\in \mathcal{M}.
\]
By Lemma \ref{lemma-stabilizer}, $(K_1,K_2,\cdots,K_{ \mathbf{K} })$ is a stabilizer of system \eqref{system} and \eqref{Lyapunov K} admits a unique solution $(P_1,P_2,\cdots,$
$P_{ \mathbf{K} })\in(\mathcal S^n_{++})^{ \mathbf{K} }$. Moreover, plugging \eqref{eq:Km} into \eqref{Lyapunov K}, \eqref{Lyapunov K} becomes CGAREs \eqref{SARE}.  
We complete the proof. 
\end{proof}

Although Lyapunov Iteration Scheme provides a constructive route to the optimal solution and establishes both stability and convergence, 
it has two main shortcomings in the implementation. First, it requires all of the system parameters.
Second, it is difficult to solve $P_k^{(i+1)}~(k \in \mathcal{M})$ from Lyapunov Recursion \eqref{Lyapunov Recursion}, because $P_1^{(i+1)},P_2^{(i+1)},\cdots,P_{\mathbf{K}}^{(i+1)}$ are fully coupled.
That motivates the development of model-free methods to solve the coupled problem. 
In the following two subsections, we propose two Q-learning algorithms. 

\subsection{On-Policy Q-Learning Algorithm for Problem (LQ-RS)}\label{on-policy}
~~~~~In this subsection, we propose an on-policy Q-learning algorithm which solves $P_k^{(i+1)}~(k \in \mathcal{M})$ without any system's information.

Now, we define a Q function as
	 \begin{equation}\label{eq:Q-function} 
	 \begin{aligned}
Q_k \left( x, u \right) 
&= \begin{bmatrix}
x \\
u
\end{bmatrix}^\top 
\mathbf{Q}_k 
\begin{bmatrix}
x \\
u
\end{bmatrix}
= \begin{bmatrix}
x \\
u
\end{bmatrix}^\top 
\begin{bmatrix}
Q_{kxx} & Q_{kxu} \\
Q_{kux} & Q_{kuu}
\end{bmatrix}
\begin{bmatrix}
x \\
u
\end{bmatrix} \\
&= x^\top \Pi(\mathbf{Q}_k)x + (u - \Gamma(\mathbf{Q}_k)x)^\top (R_k + D_k^\top P_kD_k)(u - \Gamma(\mathbf{Q}_k)x),~~k\in \mathcal{M},
\end{aligned}
\end{equation}
where $\Pi(\mathbf{Q}_k) = Q_{kxx} - Q_{kxu}(Q_{kuu})^{-1}Q_{kux}$ and $\Gamma(\mathbf{Q}_k) = -(Q_{kuu})^{-1}Q_{kux}$ with
\begin{align}
Q_{kxx}:=Q_{xx}(P_k)&=P_k + A_k^\top P_k + P_kA_k + C_k^\top P_kC_k + \sum_{j=1}^{ \mathbf{K} }\pi_{kj}P_j+N_k,\label{eq:Qxx}\\
Q_{kxu}:=Q_{xu}(P_k)&=P_kB_k + C_k^\top P_kD_k + S_k^\top, \label{eq:Qxu}\\
Q_{kuu}:=Q_{uu}(P_k)&=R_k + D_k^\top P_kD_k, \label{eq:Quu}\\ 
Q_{kux}:=Q_{ux}(P_k)&=B_k^\top P_k + D_k^\top P_kC_k + S_k.\nonumber
\end{align} 
For simplicity, we denote $\Theta(P_k^{(i)})=\Theta_k^{(i)}$, where $\Theta=Q_{xx}, Q_{xu}, Q_{ux}, Q_{uu}$ and $\Theta_k^{(i)}=Q_{kxx}^{(i)}, Q_{kxu}^{(i)}$, $Q_{kux}^{(i)}, Q_{kuu}^{(i)}$. 
If $P_k~(k\in \mathcal{M})$ is the solution of CGAREs \eqref{SARE} and $u_k=\Gamma(\mathbf{Q}_k)x$ is the optimal solution for Problem (LQ-RS), we have
\begin{equation}\label{eq:value function}
Q_k(x, u_k) = x^\top P_k x,~~~~~k\in \mathcal{M}.
\end{equation}
Then, value function is rewritten as
\begin{equation}\label{eq:new bellman equation}
\begin{aligned}
Q_k(x, u_k) =&~ \mathbb{E} \bigg\{ \int_t^{\infty} X(s)^\top \Big[ N(\alpha_s) + 2\Gamma(\mathbf{Q}(\alpha_s))^\top S(\alpha_s) \\
&\qquad~+ \Gamma(\mathbf{Q}(\alpha_s))^\top R(\alpha_s)\Gamma(\mathbf{Q}(\alpha_s)) \Big] X(s)ds\;\Big|\;X(t)=x,\alpha_t=k\bigg\},
\end{aligned} 
\end{equation}
where $\Gamma(\mathbf{Q}(\alpha_s))=\Gamma(\mathbf{Q}_k)$ when $\alpha_s=k~(k\in \mathcal{M})$.

Based on \eqref{eq:new bellman equation}, we formulate the following on-policy Q-learning algorithm to solve optimal control.
	\begin{algorithm}[H]
\caption{On-policy Q-learning algorithm for Problem (LQ-RS)}
\label{alg:2}
\begin{algorithmic}[1] 
\State \textbf{Initialization:} Select any stabilizer $(\Gamma(\mathbf{Q}_1^{(0)}), \Gamma(\mathbf{Q}_2^{(0)}), \cdots, \Gamma(\mathbf{Q}_{ \mathbf{K} }^{(0)}))$ for system \eqref{system}.
\State Let $i = 0$ and $\varepsilon > 0$.
\State \textbf{do} \{
\State Running system \eqref{system} with $u^{(i)}(s) = \sum_{k=1}^ \mathbf{K} \Gamma(\mathbf{Q}_k^{(i)}) X^{(i)}(s)\chi_{\{\alpha_s=k\}}(s)$ on $[t, \infty)$.
\State \textbf{Policy Evaluation} (Reinforcement): Solve $\mathbf{Q}^{(i+1)}_k$ $(k\in \mathcal{M})$ from the identity
\begin{equation}\label{eq:q-learning onPolicy Evaluation}
\begin{aligned}
\begin{bmatrix}
x \\
u^{(i)}_k 
\end{bmatrix}^\top 
&\mathbf{Q}^{(i+1)}_k 
\begin{bmatrix}
x \\
u^{(i)}_k
\end{bmatrix} 
	= \mathbb{E} \bigg\{ \int_t^\infty X^{(i)}(s)^\top \Big[ N(\alpha_s) + 2\Gamma(\mathbf{Q}^{(i)}(\alpha_s))^\top S(\alpha_s) \\
& + \Gamma(\mathbf{Q}^{(i)}(\alpha_s))^\top R(\alpha_s)\Gamma(\mathbf{Q}^{(i)}(\alpha_s)) \Big] X^{(i)}(s)ds \;\bigg|\; X^{(i)}(t) = x, \alpha_t = k \bigg\},~k\in \mathcal{M}.
\end{aligned} 
\end{equation} 
\State \textbf{Policy Improvement} (Update): Update $\Gamma(\mathbf{Q}_k^{(i+1)})$ by
\begin{equation}\label{eq:q-learning Policy Improvement}
\Gamma(\mathbf{Q}_k^{(i+1)}) = -(Q^{(i+1)}_{kuu})^{-1} (Q^{(i+1)}_{kux}),~~~~~~~k\in \mathcal{M}. 
\end{equation} 
\State $i \leftarrow i + 1$.
\State \} \textbf{until} $\|\mathbf{Q}^{(i+1)}_k - \mathbf{Q}^{(i)}_k\| < \varepsilon$ for all $k\in \mathcal{M}$.
\end{algorithmic}
\end{algorithm}

By virtue of the introduction of Q-function, Algorithm \ref{alg:2} does not require the knowledge of the system dynamics, which means that it is a totally model-free algorithm. In addition, Algorithm \ref{alg:2} overcomes the difficulty of
the randomness of $P_k^{(i+1)}~(k\in \mathbf{M})$ at the future time $t+\Delta t$ and decouples $\{P^{(i+1)}_k\}_{k=1}^{\mathbf{K}}$
by system stability properties over the entire time interval $[t,+\infty)$.

\begin{lemma}\label{lemma-1}
Algorithm \ref{alg:2} is equivalent to Lyapunov Iteration Scheme \eqref{Lyapunov Recursion}-\eqref{algorithm 1-improvement}.
\end{lemma}

\begin{proof}
We apply generalized It\^o's formula to $X^{(i)}(s)^\top P^{(i+1)}(\alpha_s)X^{(i)}(s)$ by Lemma \ref{lem:Ito}. Suppose 
$P^{(i+1)}(\alpha_s)$ is the solution of Lyapunov Recursion \eqref{Lyapunov Recursion}, and $X^{(i)}(s)$ is the state in \eqref{system}. 
We have
\begin{equation}\label{Ito-1}
\begin{aligned}
&\quad d\left[X^{(i)}(s)^\top P^{(i+1)}(\alpha_s)X^{(i)}(s)\right]\\
&=\bigg\{X^{(i)}(s)^\top\Big[\big(A(\alpha_s)+B(\alpha_s)K^{(i)}(\alpha_s)\big)^\top P^{(i+1)}(\alpha_s)+P^{(i+1)}(\alpha_s)\big(A(\alpha_s)+B(\alpha_s)K^{(i)}(\alpha_s)\big)\\
&\quad+\big(C(\alpha_s)+D(\alpha_s)K^{(i)}(\alpha_s)\big)^\top P^{(i+1)}(\alpha_s)\big(C(\alpha_s)+D(\alpha_s)K^{(i)}(\alpha_s)\big)\\
&\quad+\sum_{j=1}^{ \mathbf{K} }\pi_{\alpha_s j}P_j^{(i+1)}\Big]X^{(i)}(s)\bigg\}ds +\left\{...\right\}dW(s).
\end{aligned}
\end{equation}
Integrating \eqref{Ito-1} on $[t, \infty)$ and taking expectations on both sides, one gets
\begin{equation}\label{integrating}
\begin{aligned}
&\quad x^\top P_k^{(i+1)}x\\
&=-\mathbb E\bigg\{\int_t^{\infty}X^{(i)}(s)^\top\Big[\big(A(\alpha_s)+B(\alpha_s)K^{(i)}(\alpha_s)\big)^\top P^{(i+1)}(\alpha_s)+P^{(i+1)}(\alpha_s)\big(A(\alpha_s)+B(\alpha_s)K^{(i)}(\alpha_s)\big)\\
&\quad+\big(C(\alpha_s)+D(\alpha_s)K^{(i)}(\alpha_s)\big)^\top P^{(i+1)}(\alpha_s)\big(C(\alpha_s)+D(\alpha_s)K^{(i)}(\alpha_s)\big)\\
&\quad+\sum_{j=1}^{ \mathbf{K} }\pi_{\alpha_s j}P_j^{(i+1)}\Big]X^{(i)}(s)ds\Big|X^{(i)}(t)=x,\alpha_t=k\bigg\}\\
&=-\mathbb E\bigg\{\sum_{k=1}^ \mathbf{K} \int_t^{\infty}X^{(i)}(s)^\top\Big[\big(A_k+B_kK_k^{(i)}\big)^\top P_k^{(i+1)}+P_k^{(i+1)}\big(A_k+B_kK_k^{(i)}\big)+\sum_{j=1}^{ \mathbf{K} }\pi_{k j}P_j^{(i+1)}\\
&\quad+\big(C_k+D_kK_k^{(i)}\big)^\top P_k^{(i+1)}\big(C_k+D_kK_k^{(i)}\big)\Big]X^{(i)}(s)\chi_{\{\alpha_s=k\}}(s)ds\Big|X^{(i)}(t)=x,\alpha_t=k\bigg\}.\\
\end{aligned}
\end{equation}
Taking Lyapunov Recursion \eqref{Lyapunov Recursion} into \eqref{integrating} yields 
\begin{equation}
\begin{aligned}\label{eq:xpx}
x^\top P_k^{(i+1)}x
&=\mathbb E\bigg\{\int_t^{\infty}X^{(i)}(s)^\top\Big[N(\alpha_s)+2K^{(i)}(\alpha_s)^{\top} S(\alpha_s)\\
&\quad +K^{(i)}(\alpha_s)^{\top}R(\alpha_s)K^{(i)}(\alpha_s)\Big]X^{(i)}(s)ds\Big|X^{(i)}(t)=x,\alpha_t=k\bigg\},~~k\in \mathcal{M}.
\end{aligned}
\end{equation}
On the other hand, 
if $P_k^{(i+1)}\in\mathcal S^n_{++}$ $(k\in \mathcal{M})$ is the solution of \eqref{eq:xpx}, for any $t>0$, integrating \eqref{Ito-1} on $[t, t+\Delta t]$ and taking expectations on both sides, one gets
\begin{equation}
	\label{combining-1}
	\begin{aligned}
		&~~~\mathbb E\big[X^{(i)}(t+\Delta t)^\top P^{(i+1)}(\alpha_{t+\Delta t})X^{(i)}(t+\Delta t)\big|X^{(i)}(t)=x,\alpha_t=k\big]-x^\top P_k^{(i+1)}x\\
&=\mathbb E\Big\{\int_t^{t+\Delta t}X^{(i)} (s)^\top\Big[\big(A(\alpha_s)+B(\alpha_s)K^{(i)}(\alpha_s)\big)^\top P^{(i+1)}(\alpha_s)+P^{(i+1)}(\alpha_s)\big(A(\alpha_s)+B(\alpha_s)K^{(i)}(\alpha_s)\big)\\
&~~~+\big(C(\alpha_s)+D(\alpha_s)K^{(i)}(\alpha_s)\big)^\top P^{(i+1)}(\alpha_s)\big(C(\alpha_s)+D(\alpha_s)K^{(i)}(\alpha_s)\big)\\
&~~~+\sum_{j=1}^{ \mathbf{K} }\pi_{\alpha_s j}P_j^{(i+1)}\Big]X^{(i)}(s)ds\Big|X^{(i)}(t)=x,\alpha_t=k\Big\},\qquad\qquad\qquad\qquad~~~k\in \mathcal{M}.
	\end{aligned}
\end{equation} 
From \eqref{eq:xpx}, we obtain
\begin{equation}\label{combining-2}
	\begin{aligned}
x^\top P_k^{(i+1)}x &=\mathbb E\Big\{\int_t^{t+\Delta t}X^{(i)}(s)^\top\Big[N(\alpha_s)+2K^{(i)}(\alpha_s)^{\top} S(\alpha_s)\\
&~~~+K^{(i)}(\alpha_s)^{\top} R(\alpha_s)K^{(i)}(\alpha_s)\Big]X^{(i)}(s)ds\Big|X^{(i)}(t)=x,\alpha_t=k\Big\}\\
&~~~+\mathbb E\big[X^{(i)}(t+\Delta t)^\top P^{(i+1)}(\alpha_{t+\Delta t})X^{(i)}(t+\Delta t)\big|X^{(i)}(t)=x,\alpha_t=k\big],~~~k\in \mathcal{M}.
	\end{aligned}
\end{equation}
Combining \eqref{combining-1} and \eqref{combining-2}, we have, for $k\in\mathcal{M}$, 
\begin{equation}\label{combining}
	\begin{aligned}
		&\mathbb E\Big\{\int_t^{t+\Delta t}X^{(i)} (s)^\top\Big[\big(A(\alpha_s)+B(\alpha_s)K^{(i)}(\alpha_s)\big)^\top P^{(i+1)}(\alpha_s)+P^{(i+1)}(\alpha_s)\big(A(\alpha_s)+B(\alpha_s)K^{(i)}(\alpha_s)\big)\\
&~~~+\big(C(\alpha_s)+D(\alpha_s)K^{(i)}(\alpha_s)\big)^\top P^{(i+1)}(\alpha_s)\big(C(\alpha_s)+D(\alpha_s)K^{(i)}(\alpha_s)\big)+\sum_{j=1}^{ \mathbf{K} }\pi_{\alpha_s j}P_j^{(i+1)}\\
&~~~+N(\alpha_s)+2K^{(i)}(\alpha_s)^{\top} S(\alpha_s)+K^{(i)}(\alpha_s)^{\top} R(\alpha_s)K^{(i)}(\alpha_s)\Big]X^{(i)}(s)ds\Big|X^{(i)}(t)=x,\alpha_t=k\Big\}=0.\\ 
	\end{aligned}
\end{equation}
Dividing $\Delta t$ on both sides of \eqref{combining}, let $\Delta t\rightarrow 0$, one gets 
\begin{equation}\label{Lyapunov x}
\begin{aligned}
&x^\top\Big[(A_k+B_kK_k^{(i)})^\top P_k^{(i+1)}+P_k^{(i+1)}(A_k+B_kK_k^{(i)})+(C_k+D_kK_k^{(i)})^\top P_k^{(i+1)}(C_k+D_kK_k^{(i)})\\
&~~~~\qquad+\sum_{j=1}^{ \mathbf{K} }\pi_{k j}P_j^{(i+1)}+K_k^{(i)\top} R_kK_k^{(i)}+S_k^\top K_k^{(i)}+K_k^{(i)\top} S_k+N_k\Big]x=0,~~~k\in \mathcal{M},
\end{aligned}
\end{equation}
which holds true for any $x\in\mathbb R^n$, and $P_k^{(i+1)}\in\mathcal S^n_{++}~(k\in \mathcal{M})$ satisfies Lyapunov Recursion \eqref{Lyapunov Recursion}.
So, \eqref{eq:xpx} is equivalent to Lyapunov Recursion \eqref{Lyapunov Recursion}.

Rewriting \eqref{Lyapunov x} with $u_k^{(i)}=K_k^{(i)}x~(k\in \mathcal{M})$, yields
\begin{equation}\label{Lyapunov x matrix}
\begin{bmatrix}
x \\
u_k^{(i)}
\end{bmatrix}^\top
\begin{bmatrix}
Q_{kxx}^{(i+1)}-P^{(i+1)}_k& Q_{kxu}^{(i+1)} \\ 
Q_{kux}^{(i+1)} & Q_{kuu}^{(i+1)}
\end{bmatrix}
\begin{bmatrix}
x \\
u_k^{(i)}
\end{bmatrix}=0,~~k\in \mathcal{M}. 
\end{equation}
Adding \eqref{Lyapunov x matrix} to \eqref{eq:xpx}, we obtain
\begin{equation*}
\begin{aligned}
&\quad \begin{bmatrix}
x \\
u_k^{(i)}
\end{bmatrix}^\top
\begin{bmatrix}
Q_{kxx}^{(i+1)}-P^{(i+1)}_k& Q_{kxu}^{(i+1)} \\ 
Q_{kux}^{(i+1)} & Q_{kuu}^{(i+1)}
\end{bmatrix}
\begin{bmatrix}
x \\
u_k^{(i)}
\end{bmatrix}+
\begin{bmatrix}
x \\
u_k^{(i)}
\end{bmatrix}^\top
\begin{bmatrix}
P_k^{(i+1)} & 0 \\ 
0 & 0
\end{bmatrix}
\begin{bmatrix}
x \\
u_k^{(i)}
\end{bmatrix}\\
&= \mathbb{E} \bigg\{ \int_t^\infty X^{(i)}(s)^\top \Big[ N(\alpha_s) 
+ 2\Gamma(\mathbf{Q}^{(i)}(\alpha_s))^\top S(\alpha_s) \\
&\quad+ \Gamma(\mathbf{Q}^{(i)}(\alpha_s))^\top R(\alpha_s)\Gamma(\mathbf{Q}^{(i)}(\alpha_s)) \Big] X^{(i)}(s)ds \bigg| X^{(i)}(t) = x, \alpha_t = k \bigg\},~~k\in \mathcal{M}, 
\end{aligned}
\end{equation*}
which confirms \eqref{eq:q-learning onPolicy Evaluation}. Since \eqref{eq:q-learning onPolicy Evaluation} is equivalent to \eqref{eq:xpx}, and \eqref{eq:xpx} is equivalent to Lyapunov Recursion \eqref{Lyapunov Recursion}, 
then \eqref{eq:q-learning onPolicy Evaluation} is equivalent to Lyapunov Recursion \eqref{Lyapunov Recursion}.

Moreover, it is easy to verify that
\begin{equation}\label{algorithm2 K}
\begin{aligned}
\Gamma(\mathbf{Q}_k^{(i+1)})&=-(Q^{(i+1)}_{kuu})^{-1}{Q^{(i+1)}_{kux}}\\
&=-\Big(R_k + D_k^\top P_k^{(i+1)}D_k \Big)^{-1}\Big(B_k^\top P_k^{(i+1)} + D_k^\top P_k^{(i+1)}C_k + S_k\Big)\\
&=K_k^{(i+1)},~k\in \mathcal{M}.
\end{aligned}
\end{equation}
We complete the proof.
\end{proof}

Based on the equivalence, we show that Algorithm \ref{alg:2} also has stabilization and convergence properties.

\begin{theorem}
Assume Assumption \ref{ass1} holds. Let $(\Gamma(\mathbf{Q}_1^{(0)}),\Gamma(\mathbf{Q}_2^{(0)}),...,\Gamma(\mathbf{Q}_{ \mathbf{K} }^{(0)}))$ 
be the known initial stabilizer for system \eqref{system}, then all the control gains $\{\Gamma(\mathbf{Q}_k^{(i)})\}_{i=1}^\infty~(k\in \mathcal{M})$ updated by \eqref{eq:q-learning Policy Improvement} are stabilizers. 
Moreover, $\{\mathbf{Q}_k^{(i+1)}\}_{i=0}^{\infty}~(k\in \mathcal{M})$ in \eqref{eq:q-learning onPolicy Evaluation} converges to $\mathbf{Q}_k~(k\in \mathcal{M})$ and $\{\Gamma(\mathbf{Q}_k^{(i)})\}_{i=1}^\infty~(k\in \mathcal{M})$ converges to $\Gamma(\mathbf{Q}_k)~(k\in \mathcal{M})$.
The optimal control is $u^*=\sum_{k=1}^{ \mathbf{K} }\Gamma(\mathbf{Q}_k)X^*$ $\times \chi_{\{\alpha_s=k\}}(s)$, where $X^*$ is the corresponding optimal state with $u^*$ of system \eqref{system}.
\end{theorem}
\begin{proof}
From Theorems \ref{theorem 1} and \ref{theorem 2}, combining Lemma \ref{lemma-1}, we obtain the conclusion.
\end{proof}

Next, we show how to implement Algorithm \ref{alg:2}.
To solve the parameters of $\mathbf{Q}^{(i+1)}_k$ $(k\in \mathcal{M})$ in~\eqref{eq:q-learning onPolicy Evaluation} and $\Gamma(\mathbf{Q}^{(i+1)}_k)~(k\in \mathcal{M})$ in~\eqref{eq:q-learning Policy Improvement}, 
from \cite{Li-Li-Xu 2025}, we randomly choose \(\mathbf{H}\) initial states \(x_h\in\mathbb{R}^{n}\) and $ \mathbf{K} $ initial states \(k\in\mathcal{M}\) 
to generate the corresponding state trajectories \(X_{h,k}^{(i)}(s)=X^{(i)}(s;t,x_h,k)\) over the horizon \([t,\infty)\) with \(h = 1,2,\ldots,\mathbf{H}\).
In addition, we calculate the conditional expectation of \(\mathbf{N}\) sample paths \(X_{h,k,\nu}^{(i)}\) \((\nu=1,2,...,\mathbf{N})\) 
under the known initial $(t,x_h,k)$. Moreover, we collect trajectory samples at times \(t_l\) (\(t_l \geq 0\)) and adopt a step size $\delta_{t_l}=t_{l+1}-t_{l}$, where \(l = 1,2,\ldots,\mathbf{L}\), and \(\mathbf{L}\) is sufficiently large.
	We denote
	\begin{equation*}
L^{(i)}_{xu,k} = \begin{bmatrix}
\bar{x}_1 & 2x_1^\top \otimes {u^{(i)}_{1,k}}^\top & \bar{u}^{(i)}_{1,k} \\
\bar{x}_2 & 2x_2^\top \otimes {u^{(i)}_{2,k}}^\top & \bar{u}^{(i)}_{2,k} \\
\vdots&\vdots& \vdots\\
\bar{x}_{\mathbf{H}} & 2x_{\mathbf{H}}^\top \otimes {u^{(i)}_{\mathbf{H},k}}^\top & \bar{u}^{(i)}_{\mathbf{H},k}
\end{bmatrix}, ~~k\in \mathcal{M},
\end{equation*}
where $u_{h,k}^{(i)}=\Gamma(\mathbf{Q}_k^{(i)})x_h+\omega$ and $\omega$ is a persistent excitation signal. We denote 
\begin{equation*}
	\Psi^{(i)}_{k}= \begin{bmatrix}
	\frac{1}{\mathbf{N}} \sum_{\nu=1}^{\mathbf{N}} \left[ \sum_{l=1}^{\mathbf{L}} {(X^{(i)}_{1,k,\nu}(t_l))}^\top \Sigma^{(i)}(\alpha_{t_l}) X^{(i)}_{1,k,\nu}(t_l) \delta_{t_l} \right] \\
\frac{1}{\mathbf{N}} \sum_{\nu=1}^{\mathbf{N}} \left[ \sum_{l=1}^{\mathbf{L}} {(X^{(i)}_{2,k,\nu}(t_l))}^\top \Sigma^{(i)}(\alpha_{t_l}) X^{(i)}_{2,k,\nu}(t_l) \delta_{t_l} \right] \\
\vdots \\
\frac{1}{\mathbf{N}} \sum_{\nu=1}^{\mathbf{N}} \left[ \sum_{l=1}^{\mathbf{L}} {(X^{(i)}_{\mathbf{H},k,\nu}(t_l))}^\top \Sigma^{(i)}(\alpha_{t_l}) X^{(i)}_{\mathbf{H},k,\nu}(t_l) \delta_{t_l} \right]
	\end{bmatrix}, ~~k\in \mathcal{M},
	\end{equation*}
	where $\Sigma^{(i)}(\alpha_{t_l})=N(\alpha_{t_l})+2K^{(i)}(\alpha_{t_l})^{\top} S(\alpha_{t_l})+{(K^{(i)}(\alpha_{t_l}))}^\top R(\alpha_{t_l})K^{(i)}(\alpha_{t_l})$. 
	We note that the number of initial states $\mathbf{H}$ must satisfy $\mathbf{H} \geq \frac{1}{2} [n(n+1) + 2nr + r(r+1)]$ to solve $Q_{kxx}^{(i+1)}$, $Q_{kux}^{(i+1)}$ and $Q_{kuu}^{(i+1)}$ with $k\in \mathcal{M}$. 

Algorithm \ref{alg:2} is required to satisfy a persistent excitation (PE) condition (Bradtke et al.\cite{S. J. Bradtke}, Al-Tamimi et al.\cite{A. Al-Tamimi}), which is exactly the following assumption.
The PE signal $\omega$ is introduced in $u_{h,k}^{(i)}=\Gamma(\mathbf{Q}_k^{(i)})x_h+\omega$ to guarantee that $L^{(i)}_{xu,k}$ has full column rank, which ensures that $\mathbf{Q}^{(i+1)}_k$ $(k\in \mathcal{M})$ can be uniquely solved. 

\begin{assumption}\label{as:3.1}
In order to guarantee that Algorithm \ref{alg:2} can be implemented online at each step, $L^{(i)}_{xu,k}$ is required to have full column rank, we assume
	there exists $\mathbf{H}_0>0$ such that for all $\mathbf{H} \geq \mathbf{H}_0$, $\mathrm{rank}(L^{(i)}_{xu,k})=\frac{1}{2} [n(n+1) + 2nr + r(r+1)]$. 
\end{assumption}

	Under Assumption \ref{as:3.1}, using least-squares method, we have
\begin{equation}\label{eq:algo2-Kronecker}
	\begin{bmatrix}
vech(Q^{(i+1)}_{kxx}) \\ vec(Q^{(i+1)}_{kxu}) \\ vech(Q^{(i+1)}_{kuu}) 
\end{bmatrix}
= \left( {L^{(i)}_{xu,k}}^\top L^{(i)}_{xu,k} \right)^{-1} {L^{(i)}_{xu,k}}^\top \Psi^{(i)}_{k},~~~k\in \mathcal{M}. 
	\end{equation}
So, \eqref{eq:q-learning onPolicy Evaluation} can be solved by \eqref{eq:algo2-Kronecker}. We also can obtain \eqref{eq:q-learning Policy Improvement}.

\subsection{Off-Policy Q-Learning Algorithm for Problem (LQ-RS)}\label{sec.3.2}
~~~~In Subsection \ref{on-policy}, we introduce the on-policy Q-learning algorithm for Problem (LQ-RS), which generates a new trajectory data pair $(X^{(i)}(s), u^{(i)}(s))$ for each iteration, 
resulting in low data utilization efficiency and increased running time. 
We now present an off-policy Q-learning algorithm for Problem (LQ-RS), 
which reuses trajectory data to improve data utilization efficiency and reduce running time.

We consider using off-policy trajectory data pair $(X(s), u(s))$ to calculate $\mathbf{Q}^{(i+1)}_k$ and $\Gamma(\mathbf{Q}_k^{(i+1)})$ for $k\in \mathcal{M}$. 
We demonstrate the following off-policy Q-learning algorithm for Problem (LQ-RS).

\begin{algorithm}[H]
\caption{Off-policy Q-learning algorithm for Problem (LQ-RS)}\label{alg:3}
\begin{algorithmic}[1] 
\State \textbf{Initialization:} Select any stabilizer $(\Gamma(\mathbf{Q}_1^{(0)}), \Gamma(\mathbf{Q}_2^{(0)}), \cdots, \Gamma(\mathbf{Q}_{ \mathbf{K} }^{(0)}))$ for system \eqref{system}.
Collect stabilizing control input trajectory $u(s)$ and obtain state trajectory $X(s)$ by running system \eqref{system} on $[t, \infty)$.
\State Let $i = 0$ and $\varepsilon > 0$.
\State \textbf{do} \{
\State \textbf{Policy Evaluation} (Reinforcement): Solve $\mathbf{Q}^{(i+1)}_k$ $(k\in \mathcal{M})$ from the identity
\begin{align}\label{eq:q-learning offPolicy Evaluation}
&\begin{bmatrix}
x \\
u_k^{(i)}
\end{bmatrix}^\top 
\mathbf{Q}^{(i+1)}_k 
\begin{bmatrix}
x \\
u_k^{(i)}
\end{bmatrix}
+\mathbb{E} \bigg\{ \int_t^\infty 2\Big(u(s)-\Gamma(\mathbf{Q}^{(i)}(\alpha_s))X(s)\Big)^\top Q_{ux}^{(i+1)}(\alpha_s)X(s) \nonumber\\
&\quad+\Big(u(s) - \Gamma(\mathbf{Q}^{(i)}(\alpha_s)) X(s)\Big)^\top Q^{(i+1)}_{uu}(\alpha_s)\Big(u(s) + \Gamma(\mathbf{Q}^{(i)}(\alpha_s)) X(s)\Big)ds \bigg| X(t) = x, \alpha_t = k \bigg\} \nonumber\\
	&= \mathbb{E} \bigg\{ \int_t^\infty 2u(s)^\top S(\alpha_s)X(s)+u(s)^\top R(\alpha_s)u(s)+X(s)^\top N(\alpha_s)X(s)ds \bigg| X(t) = x, \alpha_t = k \bigg\}, \nonumber\\
\end{align} 
where $Q_{ux}^{(i+1)}(\alpha_s)=Q_{kux}^{(i+1)}$, $Q_{uu}^{(i+1)}(\alpha_s)=Q_{kuu}^{(i+1)}$ when $\alpha_s=k~(k\in \mathcal{M})$.
\State \textbf{Policy Improvement} (Update): Update $\Gamma(\mathbf{Q}_k^{(i+1)})$ with formula \eqref{eq:q-learning Policy Improvement}.

\State $i \leftarrow i + 1$.
\State \} \textbf{until} $\|\mathbf{Q}^{(i+1)}_k- \mathbf{Q}^{(i)}_k\| < \varepsilon$ for all $k\in \mathcal{M}$.
\end{algorithmic}
\end{algorithm}

\begin{theorem}
Assume Assumption \ref{ass1} holds. Let $(\Gamma(\mathbf{Q}_1^{(0)}),\Gamma(\mathbf{Q}_2^{(0)}),...,\Gamma(\mathbf{Q}_{ \mathbf{K} }^{(0)}))$ 
be the known initial stabilizer for system \eqref{system}, then all the control gains $\{\Gamma(\mathbf{Q}_k^{(i)})\}_{i=1}^\infty~(k\in \mathcal{M})$ updated by \eqref{eq:q-learning Policy Improvement} are stabilizers. 
Moreover, $\{\mathbf{Q}_k^{(i+1)}\}_{i=0}^{\infty}~(k\in \mathcal{M})$ in \eqref{eq:q-learning offPolicy Evaluation} converges to $\mathbf{Q}_k~(k\in \mathcal{M})$ and $\{\Gamma(\mathbf{Q}_k^{(i)})\}_{i=1}^\infty~(k\in \mathcal{M})$ converges to $\Gamma(\mathbf{Q}_k)~(k\in \mathcal{M})$.
The optimal control is $u^*=\sum_{k=1}^{ \mathbf{K} }\Gamma(\mathbf{Q}_k)X^*$ $\times \chi_{\{\alpha_s=k\}}(s)$, where $X^*$ is the corresponding optimal state with $u^*$ of system \eqref{system}.
\end{theorem}
\begin{proof}
First, we prove that Algorithm \ref{alg:3} is equivalent to Lyapunov Iteration Scheme.
We rewrite \eqref{system} as
\begin{equation}\label{system1}
		\begin{aligned}
		dX(s) =\ & \Big[A(\alpha_s)X(s) + B(\alpha_s)K^{(i)}(\alpha_s) X(s) + B(\alpha_s)\big(u(s) - K^{(i)}(\alpha_s) X(s)\big)\Big] ds \\
		& + \Big[C(\alpha_s)X(s) + D(\alpha_s) K^{(i)}(\alpha_s) X(s) + D(\alpha_s)\big(u(s) - K^{(i)}(\alpha_s) X(s)\big)\Big] dW(s).
	 \end{aligned}
	\end{equation}
	We apply generalized It\^o's formula to $X(s)^\top P^{(i+1)}(\alpha_s)X(s)$ by Lemma \ref{lem:Ito}, where $P^{(i+1)}(\alpha_s)$ is the solution of Lyapunov Recursion \eqref{Lyapunov Recursion},
and $X(s)$ is the state in \eqref{system}.
We have
		\begin{align}\label{eq:xpx1}
		&\quad d\Big[X(s)^\top P^{(i+1)}(\alpha_s) X(s)\Big] \nonumber\\
		&= X(s)^\top \Big[P^{(i+1)}(\alpha_s)\big(A(\alpha_s)+B(\alpha_s)K^{(i)}(\alpha_s)\big) +(A(\alpha_s)+B(\alpha_s)K^{(i)}(\alpha_s))^\top P^{(i+1)}(\alpha_s)\Big]X(s)ds \nonumber\\
		&\quad +2X(s)^\top P^{(i+1)}(\alpha_s)B(\alpha_s)(u(s)-K^{(i)}(\alpha_s)X(s))ds+\Big[(C(\alpha_s)+D(\alpha_s)K^{(i)}(\alpha_s))X(s) \nonumber\\
		&\quad+D(\alpha_s)(u(s)-K^{(i)}(\alpha_s)X(s))\Big]^\top P^{(i+1)}(\alpha_s)\Big[(C(\alpha_s)+D(\alpha_s)K^{(i)}(\alpha_s))X(s) \nonumber\\
		&\quad +D(\alpha_s)(u(s)-K^{(i)}(\alpha_s)X(s))\Big]ds+ X(s)^\top \sum_{j=1}^{ \mathbf{K} } \pi_{\alpha_s j}P_j^{(i+1)}X(s)ds+\{ \ldots \} dW(s).
	 \end{align}	 
	Then, integrating both sides of~\eqref{eq:xpx1} and taking conditional expectation from $t$ to $\infty$, we obtain
		\begin{align}\label{eq:xpx2}
		&\quad x^\top P_k^{(i+1)} x \notag\\
&=-\mathbb{E}\bigg\{ \sum_{k=1}^ \mathbf{K} \int_t^{\infty}\bigg[X(s)^\top\Big[P_k^{(i+1)}(A_k+B_k K^{(i)}_k)+(A_k+B_kK^{(i)}_k)^\top P^{(i+1)}_k \notag\\
&\quad+(C_k+D_kK^{(i)}_k)^\top P^{(i+1)}_k(C_k+D_kK^{(i)}_k) +\sum_{j=1}^{ \mathbf{K} } \pi_{k j}P_j^{(i+1)}\Big] X(s)\notag\\
		&\quad+2X(s)^\top P^{(i+1)}_kB_k\Big(u(s)-K^{(i)}_kX(s)\Big) +2X(s)^\top\Big(C_k+D_kK^{(i)}_k\Big)^\top P^{(i+1)}_kD_k\Big(u(s)-K^{(i)}_kX(s)\Big)\notag\\
		&\quad+\Big(u(s)-K^{(i)}_kX(s)\Big)^\top D_k^\top P^{(i+1)}_kD_k\Big(u(s)-K^{(i)}_kX(s)\Big)\bigg]\chi_{\{\alpha_s=k\}}(s)ds\bigg|X(t)=x,\alpha_t=k \bigg\},\notag\\
&~~~~~~~~~~~~~~~~~~~~~~~~~~~~~~~~~~~~~~~~~~~~~~~~~~~~~~~~~~~~~~~~~~~~~~~~~~\hspace*{\fill} k\in \mathcal{M}.
		\end{align}
Substituting \eqref{Lyapunov Recursion} into \eqref{eq:xpx2}, we obtain
		\begin{align}\label{eq:xpx3}
		&\quad x^\top P_k^{(i+1)} x + \mathbb{E} \bigg\{\int_t^{\infty} \Big[2\Big(u(s) - K^{(i)}(\alpha_s) X(s)\Big)^\top \Big(B(\alpha_s)^\top P^{(i+1)}(\alpha_s) \nonumber\\
&\quad+D(\alpha_s)^\top P^{(i+1)}(\alpha_s)C(\alpha_s)\Big) X(s)+\Big(u(s) - K^{(i)}(\alpha_s) X(s)\Big)^\top D(\alpha_s)^\top P^{(i+1)}(\alpha_s)D(\alpha_s)\Big(u(s) \nonumber\\
&\quad + K^{(i)}(\alpha_s) X(s)\Big)\Big]ds\bigg|X(t)=x,\alpha_t=k \bigg\}\nonumber\\
		& =\mathbb{E}\bigg\{\int_t^{\infty} X(s)^\top \Big[N(\alpha_s) + {K^{(i)}(\alpha_s)}^\top R(\alpha_s) K^{(i)}(\alpha_s)\nonumber\\
		&\quad +2S(\alpha_s)^\top K^{(i)}(\alpha_s)\Big] X(s) ds\bigg|X(t)=x,\alpha_t=k \bigg\}, \hspace*{\fill} k\in \mathcal{M}.
		\end{align}
Adding \eqref{Lyapunov x matrix} to \eqref{eq:xpx3} and using $u_k^{(i)}=\Gamma(\mathbf{Q}_k^{(i)})x~(k\in \mathcal{M})$, we obtain
\begin{align}\label{eq:xpx4}
&\quad \begin{bmatrix}
x \\
u_k^{(i)}
\end{bmatrix}^\top 
\mathbf{Q}^{(i+1)}_k 
\begin{bmatrix}
x \\
u_k^{(i)}
\end{bmatrix}+
\mathbb{E} \bigg\{\int_t^{\infty} \Big[2\Big(u(s) - \Gamma(\mathbf{Q}^{(i)}(\alpha_s))X(s)\Big)^\top \Big(Q^{(i+1)}_{ux}(\alpha_s)-S(\alpha_s)\Big) X(s)\nonumber\\
&\quad+\Big(u(s) - \Gamma(\mathbf{Q}^{(i)}(\alpha_s)) X(s)\Big)^\top \Big(Q^{(i+1)}_{uu}(\alpha_s)-R(\alpha_s)\Big)\Big(u(s)\nonumber\\
&\quad + \Gamma(\mathbf{Q}^{(i)}(\alpha_s)) X(s)\Big)\Big]ds\bigg|X(t)=x,\alpha_t=k \bigg\}\nonumber\\
&= \mathbb{E}\bigg\{\int_t^{\infty} X(s)^\top \Big[N(\alpha_s) + {\Gamma(\mathbf{Q}^{(i)}(\alpha_s))}^\top R(\alpha_s) \Gamma(\mathbf{Q}^{(i)}(\alpha_s))\nonumber\\
&\qquad\qquad\qquad+2S(\alpha_s)^\top \Gamma(\mathbf{Q}^{(i)}(\alpha_s))\Big] X(s) ds\bigg|X(t)=x,\alpha_t=k \bigg\}, ~~~~~~~\hspace*{\fill} k\in \mathcal{M}.
\end{align}
Deleting the same term $X(s)^\top \big[2S(\alpha_s)^\top \Gamma(\mathbf{Q}^{(i)}(\alpha_s))+{\Gamma(\mathbf{Q}^{(i)}(\alpha_s))}^\top R(\alpha_s) \Gamma(\mathbf{Q}^{(i)}(\alpha_s))\big]X(s)$ from both sides of \eqref{eq:xpx4}, we have
\eqref{eq:q-learning offPolicy Evaluation}. 
Conversely, if \(P_k^{(i+1)} \in \mathcal S_{++}^n\) \((k \in \mathcal M)\) is the solution to \eqref{eq:q-learning offPolicy Evaluation}, then similarly to the proof of Lemma \ref{lemma-1}, it is readily verified that \eqref{Lyapunov Recursion} holds.
Moreover, \eqref{algorithm 1-improvement} is equivalent to \eqref{eq:q-learning Policy Improvement} from Lemma \ref{lemma-1}.
Therefore, Algorithm \ref{alg:3} is equivalent to Lyapunov Iteration Scheme.

Then, from Theorems \ref{theorem 1} and \ref{theorem 2}, we obtain the conclusion.
\end{proof}

Next, we show the implementation of Algorithm \ref{alg:3}.
To solve the parameters of $\mathbf{Q}^{(i+1)}_k$ $(k\in \mathcal{M})$ in~\eqref{eq:q-learning offPolicy Evaluation} of Algorithm \ref{alg:3} and $\Gamma(\mathbf{Q}^{(i+1)}_k)~(k\in \mathcal{M})$ in~\eqref{eq:q-learning Policy Improvement}, 
we generate the state trajectory $X_{h,k}(s)=X(s;t,x_h,k)$, under the corresponding control input trajectory $u_{h,k}(s)=u(s;t,x_h,k)$ 
over horizon \([t,\infty)\) with \(h = 1,2,\ldots,\mathbf{H}\). 
In addition, we take $u_{h,k}^{(i)}=\Gamma(\mathbf{Q}_k^{(i)})x_h~(k\in \mathcal{M})$ in Algorithm \ref{alg:3}. We calculate the conditional expectation of \(\mathbf{N}\) sample paths \(X_{h,k,\nu}\) and \(u_{h,k,\nu}\) \((\nu=1,2,...,\mathbf{N})\). 
Moreover, we adopt $\delta_t$ step size, and the times at which the Markov chain state \(k\) is equal to $j$ are denoted as \(t_l^{j}\) (\(0\leq t_l^{j}\)), 
where \(l = 1,2,\ldots,\mathbf{L}\), and \(\mathbf{L}\) is sufficiently large. We denote
\begin{equation*}
I_{xx} = \begin{bmatrix}
\bar{x}_{1} \\
\bar{x}_{2} \\
\vdots \\
\bar{x}_{\mathbf{H}}
\end{bmatrix}, \quad\quad\quad
\delta^{(i)}_{ux,k} = \begin{bmatrix}
{u^{(i)\top}_{1,k}} \otimes x_1^\top \\
{u^{(i)\top}_{2,k}} \otimes x_2^\top \\
\vdots \\
{u^{(i)\top}_{\mathbf{H},k}} \otimes x_{\mathbf{H}}^\top
\end{bmatrix},\quad\quad\quad
\delta^{(i)}_{uu,k} = \begin{bmatrix}
\bar{u}^{(i)}_{1,k} \\
\bar{u}^{(i)}_{2,k} \\
\vdots \\
\bar{u}^{(i)}_{\mathbf{H},k}
\end{bmatrix},
\end{equation*}
\begin{equation*}
I_{uu}^{kj} = \begin{bmatrix}
\frac{1}{\mathbf{N}} \sum_{\nu=1}^{\mathbf{N}} \left[ \sum_{l=1}^{\mathbf{L}} \bar{u}_{1,k,\nu}(t_l^{j}) \delta_t \right] \\
\frac{1}{\mathbf{N}} \sum_{\nu=1}^{\mathbf{N}} \left[ \sum_{l=1}^{\mathbf{L}} \bar{u}_{2,k,\nu}(t_l^{j}) \delta_t \right] \\
\vdots \\
\frac{1}{\mathbf{N}} \sum_{\nu=1}^{\mathbf{N}} \left[ \sum_{l=1}^{\mathbf{L}} \bar{u}_{\mathbf{H},k,\nu}(t_l^{j}) \delta_t \right]
\end{bmatrix},\quad
I_{XX}^{kj} = \begin{bmatrix}
	\frac{1}{\mathbf{N}} \sum_{\nu=1}^{\mathbf{N}} \left[ \sum_{l=1}^{\mathbf{L}} X_{1,k,\nu}(t_l^{j}) \otimes X_{1,k,\nu}(t_l^{j}) \delta_t \right] \\
\frac{1}{\mathbf{N}} \sum_{\nu=1}^{\mathbf{N}} \left[ \sum_{l=1}^{\mathbf{L}} X_{2,k,\nu}(t_l^{j}) \otimes X_{2,k,\nu}(t_l^{j}) \delta_t \right] \\
\vdots \\
\frac{1}{\mathbf{N}} \sum_{\nu=1}^{\mathbf{N}} \left[ \sum_{l=1}^{\mathbf{L}} X_{\mathbf{H},k,\nu}(t_l^{j}) \otimes X_{\mathbf{H},k,\nu}(t_l^{j}) \delta_t \right]
	\end{bmatrix}, 
\end{equation*}
\begin{equation*}
I_{Xu}^{kj} = \begin{bmatrix}
\frac{1}{\mathbf{N}} \sum_{\nu=1}^{\mathbf{N}} \left[ \sum_{l=1}^{\mathbf{L}} X_{1,k,\nu}(t_l^{j}) \otimes u_{1,k,\nu}(t_l^{j}) \delta_t \right] \\
\frac{1}{\mathbf{N}} \sum_{\nu=1}^{\mathbf{N}} \left[ \sum_{l=1}^{\mathbf{L}} X_{2,k,\nu}(t_l^{j}) \otimes u_{2,k,\nu}(t_l^{j}) \delta_t \right] \\
\vdots \\
\frac{1}{\mathbf{N}} \sum_{\nu=1}^{\mathbf{N}} \left[ \sum_{l=1}^{\mathbf{L}} X_{\mathbf{H},k,\nu}(t_l^{j}) \otimes u_{\mathbf{H},k,\nu}(t_l^{j}) \delta_t \right]
\end{bmatrix},\quad
\gamma_{k}= \begin{bmatrix}
	\frac{1}{\mathbf{N}} \sum_{\nu=1}^{\mathbf{N}} \left[ \sum_{l=1}^{\mathbf{L}} \Xi_{1,k,\nu}(\alpha_{t_l}) \delta_t \right] \\
\frac{1}{\mathbf{N}} \sum_{\nu=1}^{\mathbf{N}} \left[ \sum_{l=1}^{\mathbf{L}} \Xi_{2,k,\nu}(\alpha_{t_l}) \delta_t \right] \\
\vdots \\
\frac{1}{\mathbf{N}} \sum_{\nu=1}^{\mathbf{N}} \left[ \sum_{l=1}^{\mathbf{L}} \Xi_{\mathbf{H},k,\nu}(\alpha_{t_l}) \delta_t \right]
	\end{bmatrix} 
	\end{equation*}
	with 
\begin{align*}
\Xi_{h,k,l}(\alpha_{t_l})&={X_{h,k,\nu}(t_l)}^\top N(\alpha_{t_l})X_{h,k,\nu}(t_l)+2{u_{h,k,\nu}(t_l)}^{\top} S(\alpha_{t_l})X_{h,k,\nu}(t_l)\\
&\quad +{u_{h,k,\nu}(t_l)}^{\top} R(\alpha_{t_l})u_{h,k,\nu}(t_l),\qquad\qquad h=1,2,...,\mathbf{H},
\end{align*}
	where $k,~j\in \mathcal{M}$. We denote
\begin{equation*}
	\Phi^{(i)}= \begin{bmatrix}
I_{ \mathbf{K} } \otimes I_{xx},&\Theta^{(i)},&\zeta^{(i)}
	\end{bmatrix},\quad
\eta=\begin{bmatrix}
\gamma_{1}, &\gamma_{2}, &\cdots, &\gamma_{ \mathbf{K} }
\end{bmatrix}^\top, 
\end{equation*}
where $I_{\mathbf{K}}$ is the $\mathbf{K} \times \mathbf{K}$ identity matrix, and the block matrices $\Theta^{(i)}$ and $\zeta^{(i)}$ are defined as
\begin{equation*}
\Theta^{(i)} = \begin{bmatrix}
\Theta_{11}^{(i)} & \Theta_{12}^{(i)} & \dots & \Theta_{1\mathbf{K}}^{(i)} \\
\Theta_{21}^{(i)} & \Theta_{22}^{(i)} & \dots & \Theta_{2\mathbf{K}}^{(i)} \\
\vdots & \vdots & \ddots & \vdots \\
\Theta_{\mathbf{K}1}^{(i)} & \Theta_{\mathbf{K}2}^{(i)} & \dots & \Theta_{\mathbf{K}\mathbf{K}}^{(i)}
\end{bmatrix}, \quad
\zeta^{(i)} = \begin{bmatrix}
\zeta_{11}^{(i)} & \zeta_{12}^{(i)} & \dots & \zeta_{1\mathbf{K}}^{(i)} \\
\zeta_{21}^{(i)} & \zeta_{22}^{(i)} & \dots & \zeta_{2\mathbf{K}}^{(i)} \\
\vdots & \vdots & \ddots & \vdots \\
\zeta_{\mathbf{K}1}^{(i)} & \zeta_{\mathbf{K}2}^{(i)} & \dots & \zeta_{\mathbf{K}\mathbf{K}}^{(i)}
\end{bmatrix}
\end{equation*}
with their block elements given by
\begin{equation*}
\Theta_{kj}^{(i)} = \begin{cases}
2\delta^{(i)}_{ux,k} + 2I_{Xu}^{kk} - 2I_{XX}^{kk}(I_n \otimes \Gamma(\mathbf{Q}_k^{(i)})), & \text{if } k = j, \\
2I_{Xu}^{kj} - 2I_{XX}^{kj}(I_n \otimes \Gamma(\mathbf{Q}_j^{(i)})), & \text{if } k \neq j,
\end{cases}
\end{equation*}
and
\begin{equation*}
\zeta_{kj}^{(i)} = \begin{cases}
\delta^{(i)}_{uu,k} + I_{uu}^{kk} - I_{XX}^{kk}\bar{\Gamma}(\mathbf{Q}_k^{(i)}), & \text{if } k = j, \\
I_{uu}^{kj} - I_{XX}^{kj}\bar{\Gamma}(\mathbf{Q}_j^{(i)}), & \text{if } k \neq j.
\end{cases}
\end{equation*}
We note that $\mathbf{H}$ is required to satisfy $\mathbf{H} \geq \frac{1}{2} [n(n+1) + 2nr + r(r+1)]$ in order to solve 
$Q_{kxx}^{(i+1)}$, $Q_{kux}^{(i+1)}$ and $Q_{kuu}^{(i+1)}$ with $k\in \mathcal{M}$. 

\begin{assumption}\label{as:3.2}
In order to guarantee that Algorithm \ref{alg:3} can be implemented online at each step, $\Phi^{(i)}$ are required to have full column rank, we assume there exists $\mathbf{H}_0>0$ such that for all $\mathbf{H} \geq \mathbf{H}_0$, $\mathrm{rank}(\Phi^{(i)})=\frac{ \mathbf{K} }{2} [n(n+1) + 2nr + r(r+1)]$. 
\end{assumption} 

Under Assumption \ref{as:3.2}, using least-squares method, we have
\begin{equation}\label{eq:algo3-Kronecker}
\begin{bmatrix}
\xi_{xx}^\top, & \xi_{ux}^\top,& \xi_{uu}^\top
\end{bmatrix}^\top 
= \left( {\Phi^{(i)}}^\top \Phi^{(i)} \right)^{-1} {\Phi^{(i)}}^\top \eta,    
\end{equation}
where
\begin{equation*}
\begin{aligned}
&\xi_{xx}=[vech(Q^{(i+1)}_{1xx}),vech(Q^{(i+1)}_{2xx}),\cdots,vech(Q^{(i+1)}_{\mathbf{K} xx})], \\
&\xi_{ux}=[vec(Q^{(i+1)}_{1ux}),vec(Q^{(i+1)}_{2ux}),\cdots,vec(Q^{(i+1)}_{\mathbf{K} ux})],\\
&\xi_{uu}=[vech(Q^{(i+1)}_{1uu}),vech(Q^{(i+1)}_{2uu}),\cdots, vech(Q^{(i+1)}_{\mathbf{K} uu})].
\end{aligned}
\end{equation*}
Therefore, \eqref{eq:q-learning offPolicy Evaluation} can be solved by \eqref{eq:algo3-Kronecker}. We also can obtain \eqref{eq:q-learning Policy Improvement}.

\section{Numerical Example}\label{sec.4}
~~~~~In this section, we provide a simulation example to illustrate the proposed algorithms. 
We consider a two-dimensional linear stochastic system with two regimes. 
By setting $n=2$, $r=1$ and $ \mathbf{K} =2$, the parameter matrices of system \eqref{system} 
are given as follows
\begin{equation*}
A_1 = \begin{bmatrix}
-0.5 & 1.0 \\
0.0 & -0.3 
\end{bmatrix}, \quad
B_1 = \begin{bmatrix}
0.0 \\
1.0 
\end{bmatrix}, \quad
C_1 = \begin{bmatrix}
0.1 & 0.0 \\
0.0 & 0.1 
\end{bmatrix}, \quad 
D_1 = \begin{bmatrix}
0.05 \\
0.05 
\end{bmatrix},
\end{equation*}
\begin{equation*}
A_2 = \begin{bmatrix}
-0.4 & 0.8 \\
0.0 & -0.6 
\end{bmatrix}, \quad
B_2 = \begin{bmatrix}
1.0 \\
0.0 
\end{bmatrix}, \quad
C_2 = \begin{bmatrix}
0.1 & 0.05 \\
0.05 & 0.1 
\end{bmatrix}, \quad 
D_2 = \begin{bmatrix}
0.05 \\
0.05 
\end{bmatrix},
\end{equation*}
and the generator matrix is
\begin{equation*}
\Pi = \begin{bmatrix}
-2 & 2 \\
2 & -2
\end{bmatrix}.
\end{equation*}
The coefficients in cost functional are given as
\begin{equation*}
N_1 = \begin{bmatrix}
0.4 & 0.05 \\
0.05 & 0.2 
\end{bmatrix}, \quad 
S_1 = \begin{bmatrix}
0.1 & 0.03
\end{bmatrix}, \quad
R_1=0.12,
\end{equation*}
\begin{equation*}
N_2 = \begin{bmatrix}
0.3 & 0.04 \\
0.04 & 0.5 
\end{bmatrix}, \quad 
S_2 = \begin{bmatrix}
0.05 & 0.07 
\end{bmatrix},\quad
R_2=0.08.
\end{equation*}
We find 
\begin{equation*}
\Gamma(\mathbf{Q}^{(0)}_1) = \begin{bmatrix}
-4.41 & -0.69
\end{bmatrix}, \quad
\Gamma(\mathbf{Q}^{(0)}_2) = \begin{bmatrix}
-1.15 & 3.02
\end{bmatrix}
\end{equation*}
can drive the trajectories tend to zero when time $s$ increases. Therefore, we choose them as the initial stabilizer. 
To satisfy Assumptions \ref{as:3.1} and \ref{as:3.2}, we randomly choose more than $\frac{1}{2}[2\times(2+1)+2\times2\times1+1\times(1+1)]=6$ 
initial state vectors for each \(k\). The components of each initial state vector are independently sampled from the uniform distribution on \([0,10]\).
Here we set $\mathbf{H}=200$ and \(\mathbf{N}=50\) in this example.

(1) Algorithm \ref{alg:2} is implemented as in \eqref{eq:q-learning onPolicy Evaluation} and \eqref{eq:q-learning Policy Improvement}. At every iteration, 
we obtain \(X^{(i)} = [X^{(i)}_{(1)}, X^{(i)}_{(2)}]^\top\) and run system \eqref{system} with $u^{(i)}_k = \Gamma(\mathbf{Q}_k^{(i)}) x_h+\omega$, 
where the PE signal $\omega$ is a Gaussian white noise process with $\omega \sim \mathcal{N}(0, 8)$.
Fig. \ref{fig:convergence1} shows values of $\|\mathbf{Q}_k^{(i+1)}-\mathbf{Q}_k^{(i)}\|~(k=1,2)$ in each iteration, illustrating the variation in the differences between $\mathbf{Q}_k^{(i+1)}$ and 
$\mathbf{Q}_k^{(i)}$. We obtain that the sequence $\{\mathbf{Q}_k^{(i)}\}_{i=1}^{\infty}$ and $\{\Gamma(\mathbf{Q}_k^{(i)})\}_{i=1}^{\infty}~(k=1,2)$ converge separately to 
\begin{equation*}
\mathbf{Q}^*_1 = \begin{bmatrix}
0.2826 & 0.2008 & 0.1206\\
0.2008 & 0.4668 & 0.1918\\
0.1206 & 0.1918 & 0.1209
\end{bmatrix}, \quad
\mathbf{Q}^*_2 = \begin{bmatrix}
0.4441 & 0.1622 & 0.1665\\
0.1622 & 0.3017 & 0.0771\\
0.1665 & 0.0771 & 0.0809
\end{bmatrix},
\end{equation*}
\begin{equation*}
\Gamma(\mathbf{Q}^*_1) = \begin{bmatrix}
-0.9972 & -1.5860 
\end{bmatrix}, \quad
\Gamma(\mathbf{Q}^*_2) = \begin{bmatrix}
-2.0576 & -0.9527
\end{bmatrix}.
\end{equation*}
To check whether $\mathbf{Q}^*_k~(k=1,2)$ is the solution of CGAREs, we use \eqref{eq:Qxx}, \eqref{eq:Qxu} and \eqref{eq:Quu} to solve for the parameter $P^*_k$. We obtain
\begin{equation*}
P^*_1 = \begin{bmatrix}
0.1748 & 0.0196 \\
0.0196 & 0.1609 
\end{bmatrix}, \quad
P^*_2 = \begin{bmatrix}
0.1153 & 0.0043 \\
0.0043 & 0.2418 
\end{bmatrix}.
\end{equation*}
We define the left-hand sides of \eqref{SARE} as 
\begin{equation}\label{eq:RP}
\begin{aligned}
\mathcal{R}(P_k)&=A_k^\top P_k+P_kA_k+C_k^\top P_kC_k+N_k+\sum_{j=1}^{ \mathbf{K} }\pi_{kj}P_j\\
&~-(P_kB_k+C_k^\top P_kD_k+S_k^\top)(R_k+D_k^\top P_k D_k)^{-1}(B_k^\top P_k+D_k^\top P_kC_k+S_k ),~~k\in \mathcal{M}.
\end{aligned} 
\end{equation}
By substituting $P^*_k$ into \eqref{eq:RP}, we obtain $||\mathcal{R}(P^*_1)||=1.8990\times 10^{-2}$ and $||\mathcal{R}(P^*_2)||=1.9535\times 10^{-2}$.
System \eqref{system} is simulated separately with $\Gamma(\mathbf{Q}^{(0)}_k)$ and $\Gamma(\mathbf{Q}^{*}_k)$ ($k=1,2$) to obtain the state trajectories \(X^{(0)} = [X^{(0)}_{(1)}, X^{(0)}_{(2)}]^\top\) and \(X^{*} = [X^{*}_{(1)}, X^{*}_{(2)}]^\top\).
As shown in Fig. \ref{fig:X1 trajectory1} and Fig. \ref{fig:X2 trajectory1}, the optimal state trajectory $X^{*}$ converges to zero much faster than the initial state trajectory $X^{(0)}$ under the Markov chain $\alpha$.

(2) Algorithm \ref{alg:3} is implemented as in \eqref{eq:q-learning offPolicy Evaluation} and \eqref{eq:q-learning Policy Improvement}. 
The system state trajectory \(X = [X_{(1)}, X_{(2)}]^\top\) is generated by simulating system \eqref{system} under a stabilizing control input \(u\), as shown in Fig.~\ref{fig:X trajectory3} along with the Markov chain trajectory $\alpha$.
Fig. \ref{fig:convergence3} shows values of $\|\mathbf{Q}_k^{(i+1)}-\mathbf{Q}_k^{(i)}\|~(k=1,2)$ in each iteration, illustrating
the variation in the differences between $\mathbf{Q}_k^{(i+1)}$ and $\mathbf{Q}_k^{(i)}$. We obtain that the sequence $\{\mathbf{Q}_k^{(i)}\}_{i=1}^{\infty}$ and $\{\Gamma(\mathbf{Q}^{(i)}_k)\}_{i=1}^{\infty}~(k=1,2)$ converge separately to 
\begin{equation*}
\mathbf{Q}^*_1 = \begin{bmatrix}
0.2833 & 0.2022 & 0.1197\\
0.2022 & 0.4591 & 0.1886\\
0.1197 & 0.1886 & 0.1209
\end{bmatrix}, \quad
\mathbf{Q}^*_2 = \begin{bmatrix}
0.4430 & 0.1578 & 0.1651\\
0.1578 & 0.3089 & 0.0778\\
0.1651 & 0.0778 & 0.0809
\end{bmatrix},
\end{equation*}
\begin{equation*}
\Gamma(\mathbf{Q}^*_1) = \begin{bmatrix}
-0.9898 & -1.5712 
\end{bmatrix}, \quad
\Gamma(\mathbf{Q}^*_2) = \begin{bmatrix}
-2.0406 & -0.9617 
\end{bmatrix}.
\end{equation*}
To check whether $\mathbf{Q}^*_k~(k=1,2)$ is the solution of CGAREs, we use \eqref{eq:Qxx}, \eqref{eq:Qxu} and \eqref{eq:Quu} to solve for the parameter $P^*_k$. We obtain
\begin{equation*}
P^*_1 = \begin{bmatrix}
0.1731 & 0.0187 \\
0.0187 & 0.1591 
\end{bmatrix}, \quad
P^*_2 = \begin{bmatrix}
0.1139 & 0.0063 \\
0.0063 & 0.2373 
\end{bmatrix}.
\end{equation*}
By substituting $P^*_k$ into \eqref{eq:RP}, we obtain $||\mathcal{R}(P^*_1)||=1.0681\times 10^{-2}$ and $||\mathcal{R}(P^*_2)||=1.3227\times 10^{-2}$.
System \eqref{system} is simulated with $\Gamma(\mathbf{Q}^{*}_k)$ ($k=1,2$) to obtain the optimal state trajectory $X^* = [X^*_{(1)}, X^*_{(2)}]^\top$ as shown in
Fig. \ref{fig:X* trajectory3}, along with the Markov chain trajectory $\alpha$. 
\begin{figure*}[t]
    \centering
    \begin{subfigure}[b]{0.32\textwidth}
        \centering
        \includegraphics[width=\textwidth]{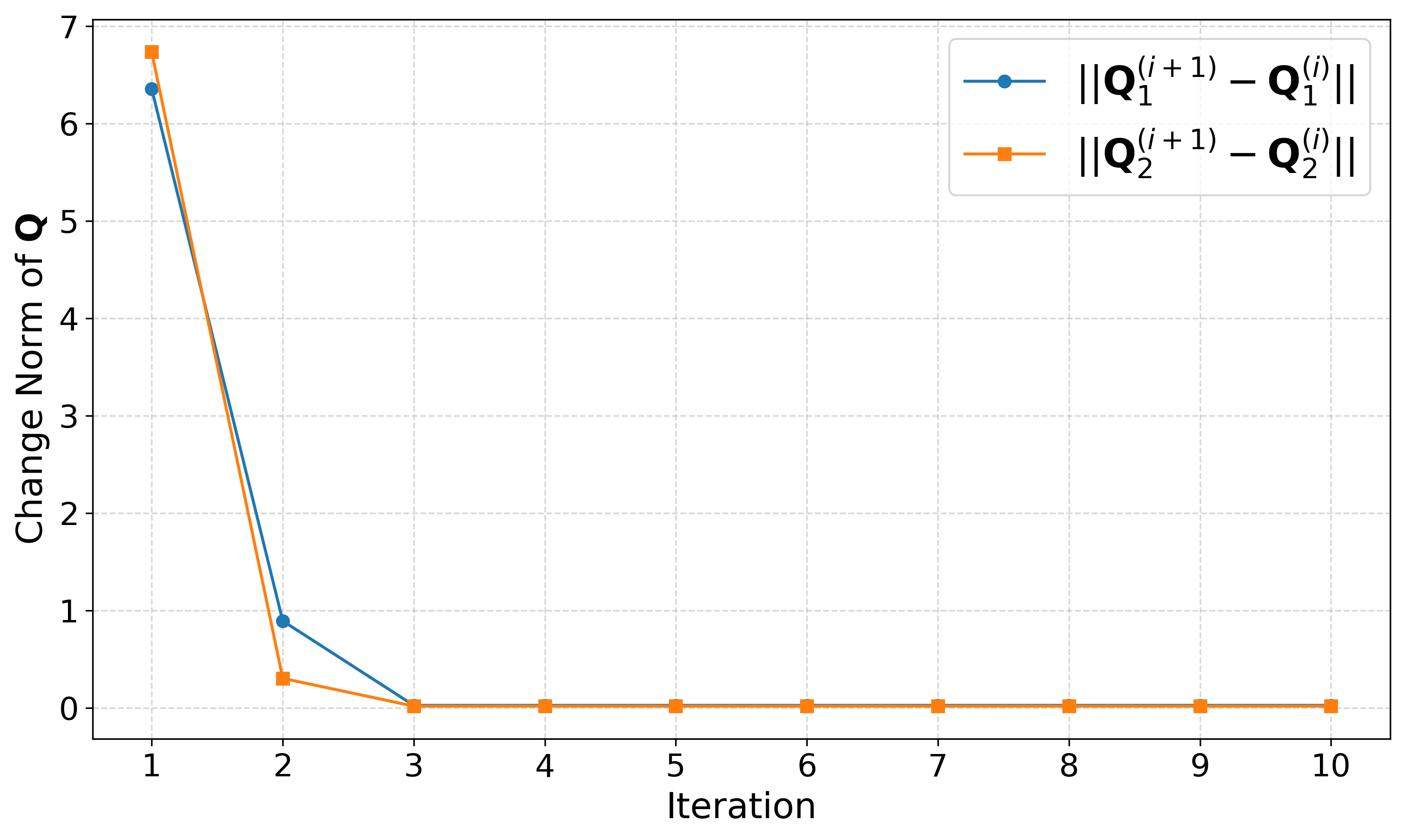}
        \caption{}
        \label{fig:convergence1}
    \end{subfigure}
    \hfill
    \begin{subfigure}[b]{0.32\textwidth}
        \centering
        \includegraphics[width=\textwidth]{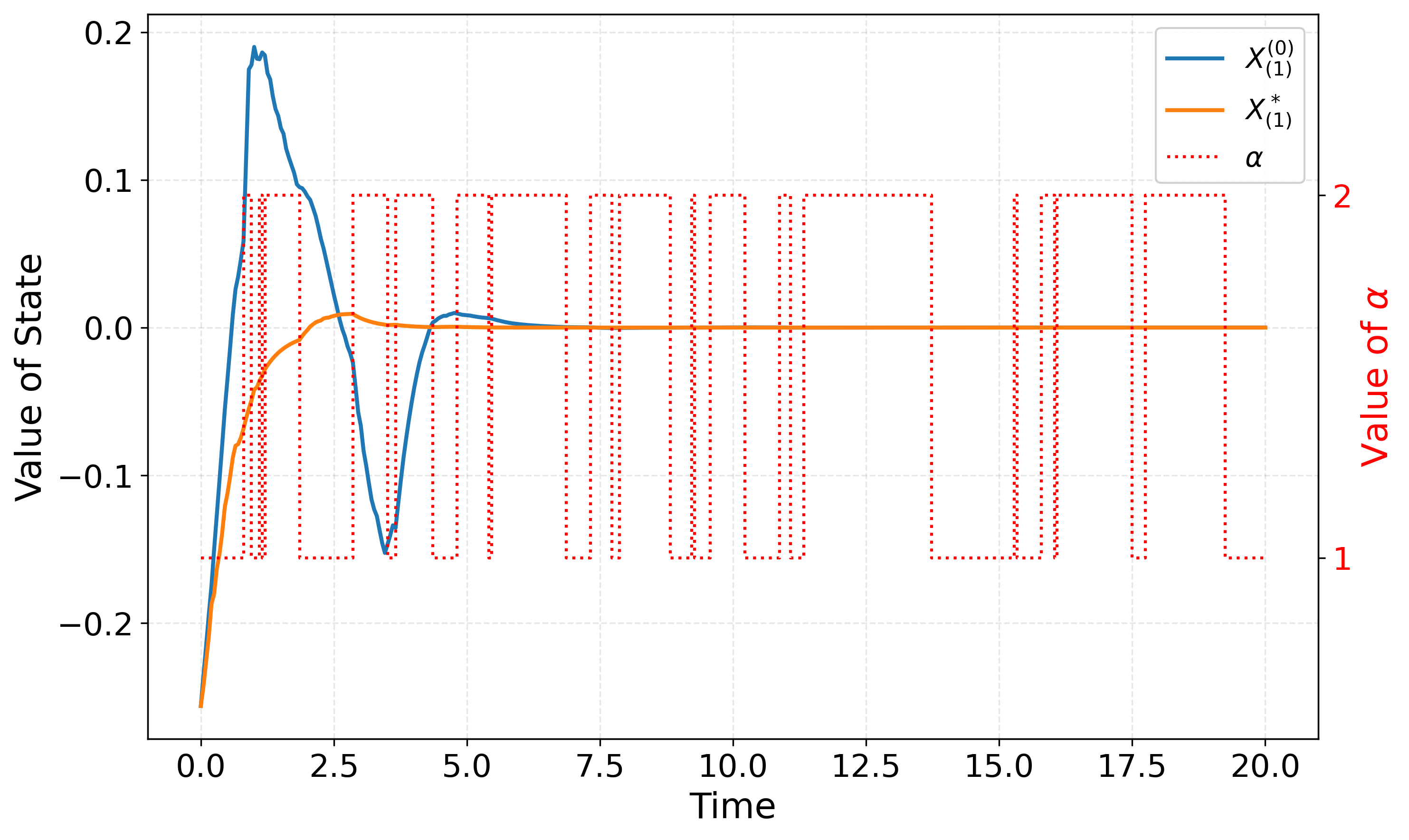}
        \caption{}
        \label{fig:X1 trajectory1}
    \end{subfigure}
    \hfill
    \begin{subfigure}[b]{0.32\textwidth}
        \centering
        \includegraphics[width=\textwidth]{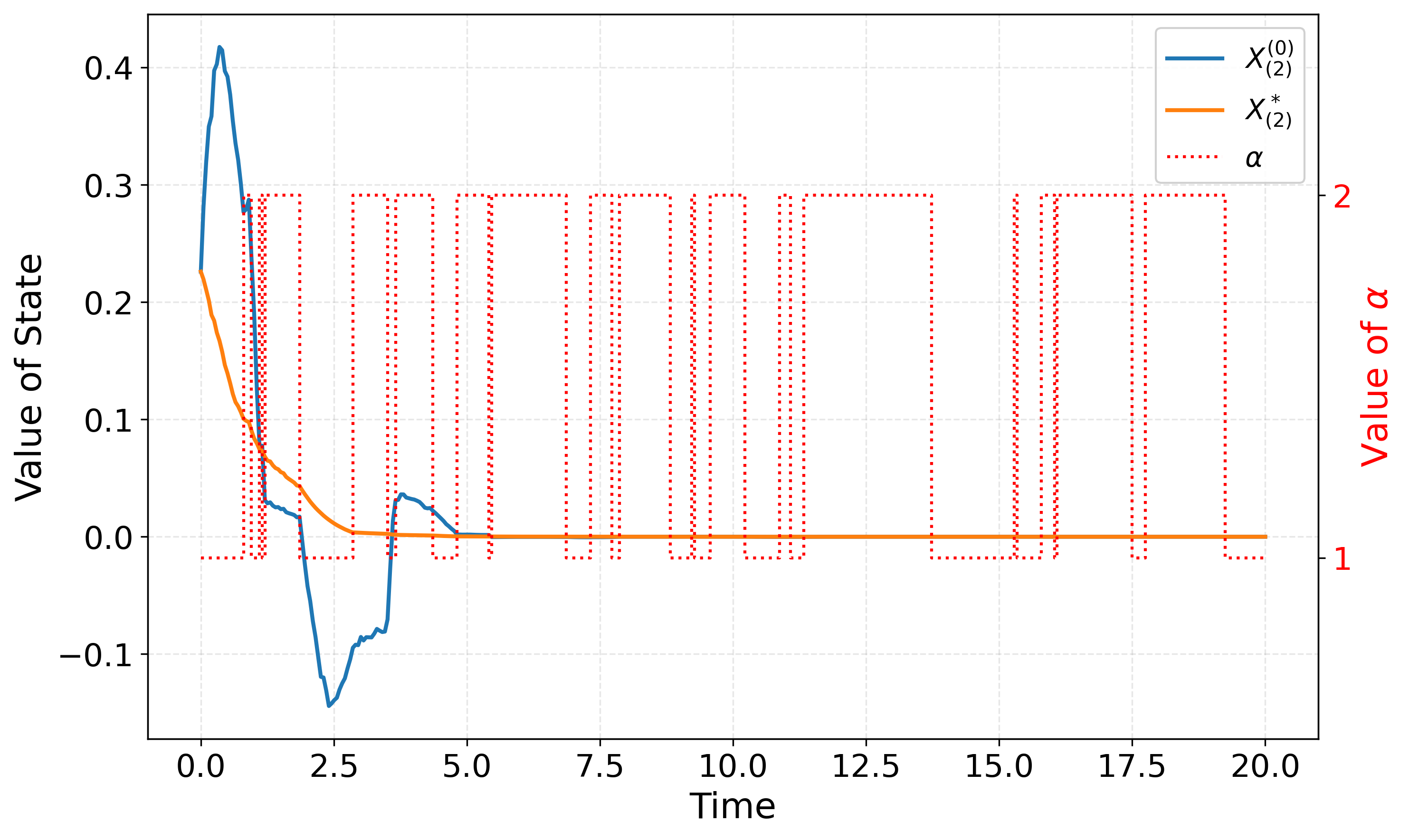}
        \caption{}
        \label{fig:X2 trajectory1}
    \end{subfigure}

    \vspace{1.5em} 

    \begin{subfigure}[b]{0.32\textwidth}
        \centering
        \includegraphics[width=\textwidth]{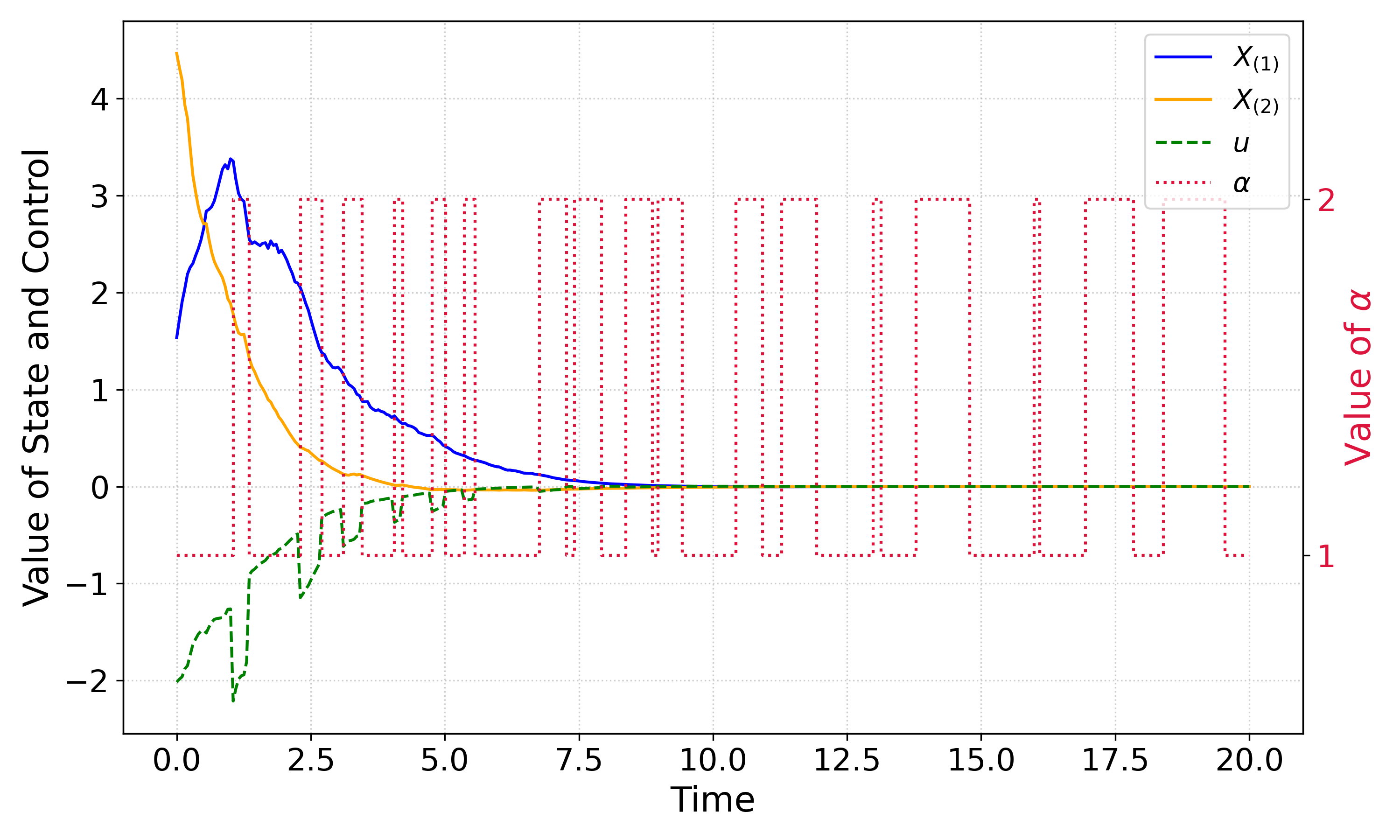}
        \caption{}
        \label{fig:X trajectory3}
    \end{subfigure}
    \hfill
    \begin{subfigure}[b]{0.32\textwidth}
        \centering
        \includegraphics[width=\textwidth]{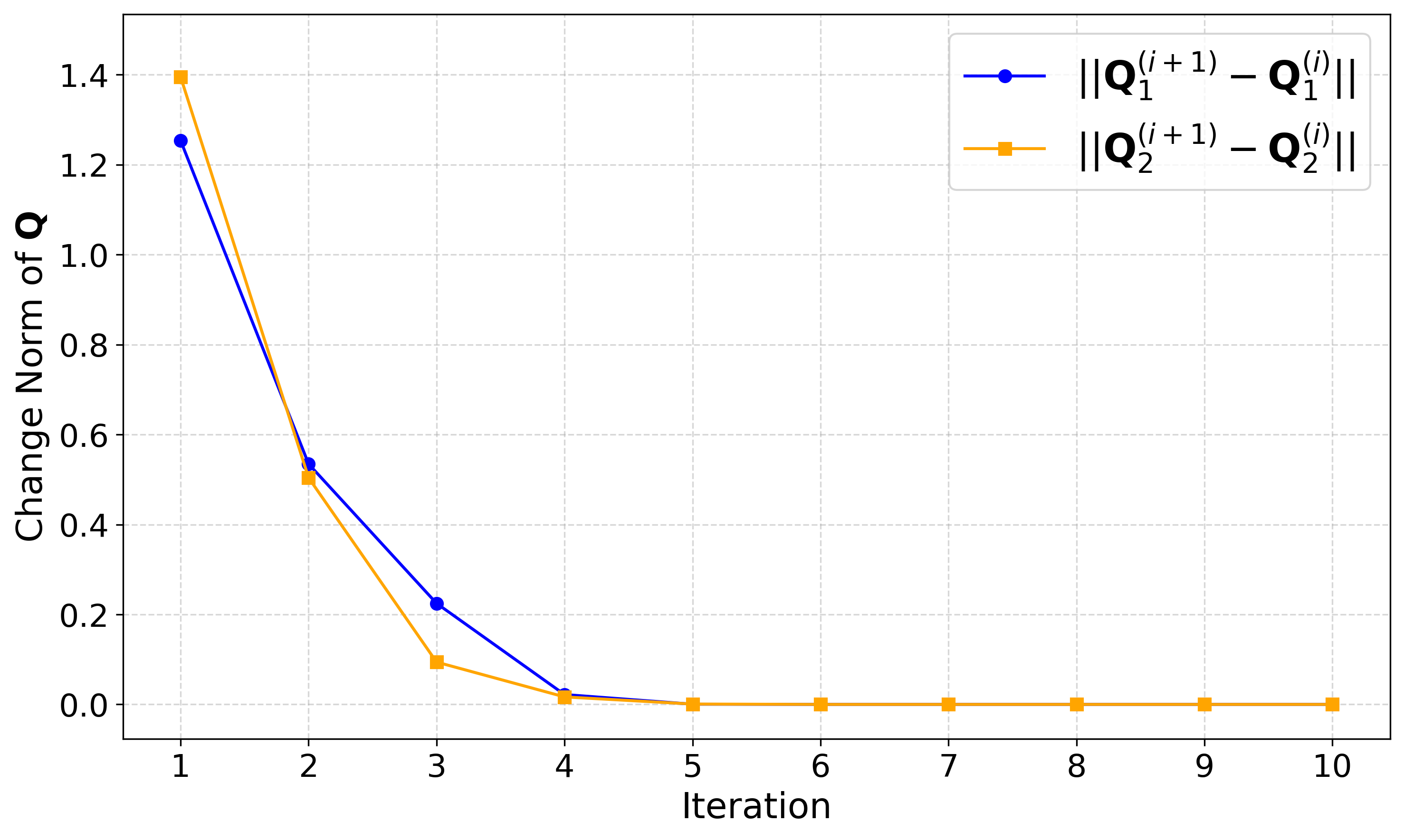}
        \caption{}
        \label{fig:convergence3}
    \end{subfigure}
    \hfill
    \begin{subfigure}[b]{0.32\textwidth}
        \centering
        \includegraphics[width=\textwidth]{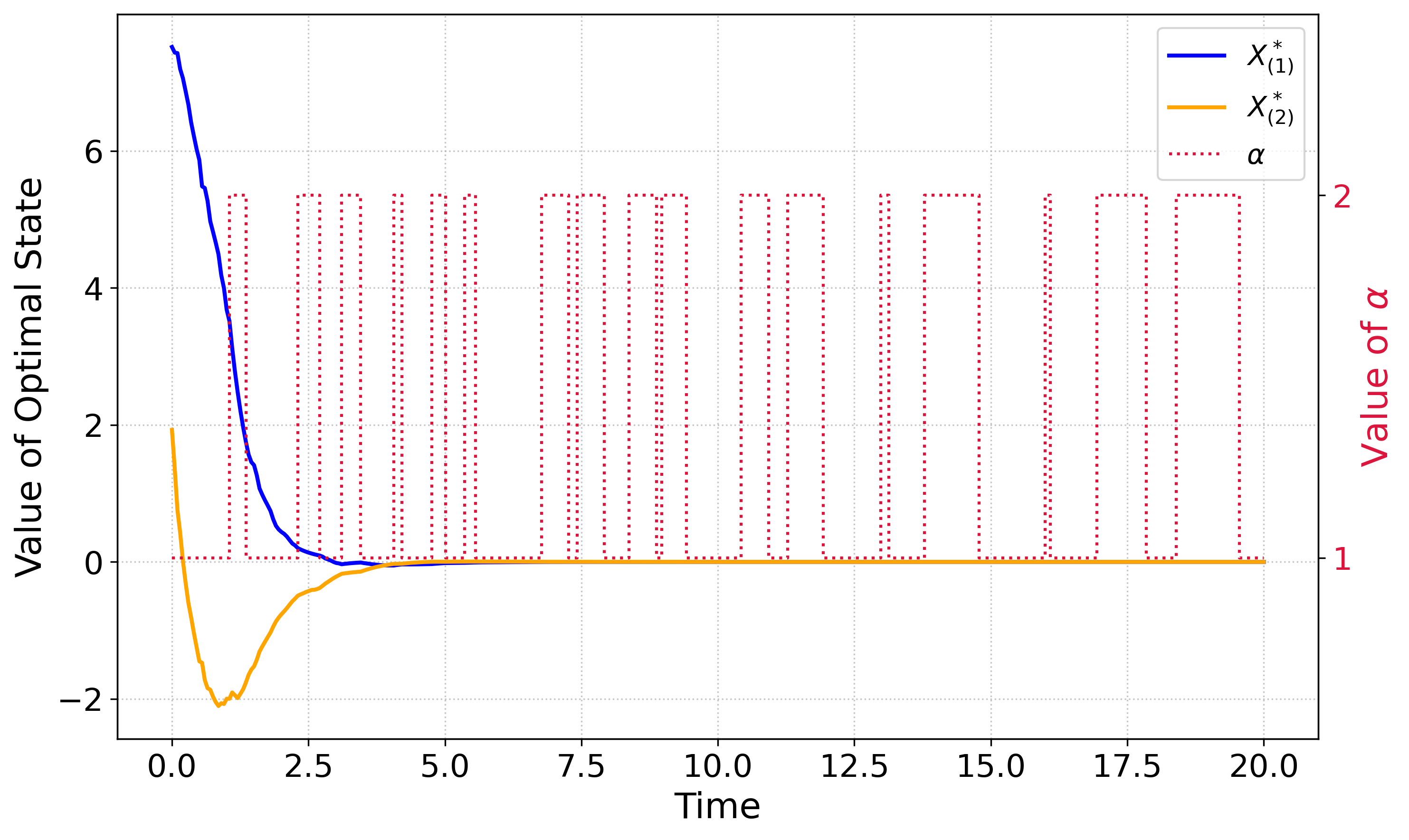}
        \caption{}
        \label{fig:X* trajectory3}
    \end{subfigure}

    \captionsetup{font=footnotesize}
    \caption{Simulation results. 
    (a)-(c) show the results for Algorithm \ref{alg:2}. 
    (a) shows the convergence of $\mathbf{Q}_k$ for $k=1,2$. 
    (b) compares $X^{(0)}_{(1)}$ and $X^{*}_{(1)}$. 
    (c) compares $X^{(0)}_{(2)}$ and $X^{*}_{(2)}$. 
    (d)-(f) show the results for Algorithm \ref{alg:3}. 
    (d) shows the trajectories of $X_{(1)}$, $X_{(2)}$, $u$, and $\alpha$. 
    (e) shows the convergence of $\mathbf{Q}_k$ for $k=1,2$. 
    (f) shows the trajectories of $X^*_{(1)}$, $X^*_{(2)}$, and $\alpha$.}
    \label{fig:overall_simulation_results}
\end{figure*}

\section{Conclusion}\label{sec.5}
~~~~In this paper, we develop a Q-learning framework for infinite-horizon continuous-time stochastic LQ optimal control problems with regime switching, effectively eliminating the need for explicit system dynamics knowledge. By exploiting the intrinsic stability properties of the system, we solve the coupled problem in infinite-horizon and propose both 
on-policy and off-policy model-free algorithms that learn optimal control gains directly from state trajectory data. Theoretical analysis establishes the equivalence, stability, and convergence of the proposed algorithms, while numerical experiments validate their practical efficacy. Some future work can be extended to the control problem with nonlinear system, $N$ players game with large population system, and so on.

\end{document}